\documentclass[A4wide,10pt]{amsart}

\usepackage{float}
\usepackage{amscd,amsmath,amssymb,amsfonts,amsthm,ascmac}
\usepackage{enumerate,mathrsfs,stmaryrd,latexsym, comment, mathtools, mathdots} 

\usepackage{amsthm, amssymb}
\usepackage[leqno]{amsmath}
\usepackage{caption}
\usepackage{subcaption}
\usepackage{hyperref}
\usepackage{relsize}
\usepackage{booktabs}
\usepackage{colortbl} %color for matrix
\usepackage{enumitem}

\usepackage[all]{xy}

\usepackage{tikz-cd}
\usepackage{tikz}
\usetikzlibrary{snakes, 
	3d, matrix,decorations.pathreplacing,calc,decorations.pathmorphing}
\usetikzlibrary {positioning}
\usetikzlibrary{matrix, fit}%color for matrix in tikz
\usetikzlibrary{backgrounds}

\usetikzlibrary{shapes,arrows,automata}

\definecolor {processblue}{cmyk}{0.96,0,0,0}

\numberwithin{equation}{section}
\makeatletter
\newcommand{\leqnomode}{\tagsleft@true\let\veqno\@@leqno}
\newcommand{\reqnomode}{\tagsleft@false\let\veqno\@@eqno}
\makeatother

	\allowdisplaybreaks

\newcommand{\defi}[1]{{\textit{#1}}}

\newcommand{\flag}{{\mathcal{F} \ell}}
\newcommand{\C}{{\mathbb{C}}}

\newcommand{\R}{{\mathbb{R}}}

\newcommand{\Z}{{\mathbb{Z}}}
\newcommand{\supp}{\mathrm{supp}}
\newcommand{\rep}{\mathrm{rep}}

\DeclareMathOperator{\GL}{GL}

\newcommand{\Cstar}{{\C^{\ast}}}

\def\a{\mathbf a}

\newcommand{\Zwp}[1]{{Z_{(#1)}'}}
\newcommand{\Zw}[1]{{Z_{(#1)}}}
\newcommand{\T}{\mathbb{T}}

\DeclareMathOperator{\Perm}{Perm}
\DeclareMathOperator{\Conv}{Conv}
\newtheorem{theorem}{Theorem}[section]
\newtheorem{lemma}[theorem]{Lemma}
\newtheorem{proposition}[theorem]{Proposition}
\newtheorem{corollary}[theorem]{Corollary}

\theoremstyle{definition}
\newtheorem{example}[theorem]{Example}
\newtheorem{definition}[theorem]{Definition}
\newtheorem{remark}[theorem]{Remark}

\begin{document}

\title[On Schubert varieties of complexity one]{On Schubert varieties of complexity one}

\date{\today}

\author[Eunjeong Lee]{Eunjeong Lee}
\address[E. Lee]{Center for Geometry and Physics, Institute for Basic Science (IBS), Pohang 37673, Republic of Korea}
\email{eunjeong.lee@ibs.re.kr}

\author[Mikiya Masuda]{Mikiya Masuda}
\address[M. Masuda]{Osaka City University Advanced Mathematical Institute \&
Department of Mathematics, Graduate School of Science, Osaka City 
University, Sumiyoshi-ku, Sugimoto, 558-8585, Osaka, Japan}
\email{masuda@osaka-cu.ac.jp}

\author[Seonjeong Park]{Seonjeong Park}
\address[S. Park]{Department of Mathematical Sciences, KAIST, Daejeon, Republic of Korea}
\email{seonjeong1124@gmail.com}



\maketitle

\begin{abstract}
Let $B$ be a Borel subgroup of $\GL_n(\C)$ and $\T$ a maximal torus contained in $B$. Then $\mathbb{T}$ acts on $\GL_{n}(\C)/B$ and every Schubert variety is $\mathbb{T}$-invariant. We say that a Schubert variety is of complexity $k$ if a maximal $\T$-orbit in $X_w$ has codimension~$k$.  In this paper, we discuss topology, geometry, and combinatorics related to Schubert varieties of complexity one.
\end{abstract}

%%%%%

\section{Introduction}

The flag manifold $\flag(\C^n)$ is the homogeneous space $\GL_n(\C)/B$, where $B$ is the set of all upper triangular matrices in $\GL_n({\C})$. The left action of $B$ on $\flag(\C^n)$ has finitely many orbits $BwB/B$, where $w$ is a permutation in $\mathfrak{S}_n$, and the Schubert variety $X_w$ is the (Zariski) closure of the $B$-orbit ${BwB/B}$. Most Schubert varieties are not smooth and they are desingularized to Bott--Samelson varieties, see~\cite{bo-sa,De74}.

Let $\T$ be the set of all diagonal matrices in $\GL_n(\C)$. Then $\T$ is isomorphic to the torus $(\C^\ast)^n$ and acts on $\flag(\C^n)$ by the left multiplication, and every Schubert variety $X_w$ is a $\mathbb{T}$-invariant irreducible subvariety of $\flag(\C^n)$. 
We say that a Schubert variety $X_w$ is of complexity $k$ with respect to the action of~$\T$ (or simply, $X_w$ \defi{is of complexity $k$}) if 
a maximal $\T$-orbit in $X_w$ has codimension~$k$. 
In this paper, we are interested in the Schubert varieties of complexity one.

There were several studies on Schubert varieties of complexity zero (i.e., toric Schubert varieties) and related combinatorics.
It is known that $X_w$ is of complexity zero if and only if a reduced decomposition of $w$ consists of distinct letters, and in this case $X_w$ is smooth and isomorphic to a Bott--Samelson  variety (\cite{Fan}, \cite{Karup}). On the other hand, permutation patterns are related to the form of reduced decompositions. A reduced decomposition of $w$ consists of distinct letters if and only if $w$ avoids the patterns $321$ and $3412$, see~\cite{Tenner2012}. 
It was also shown in~\cite{Tenner2007} that a reduced decomposition of $w$ consists of distinct letters if and only if the Bruhat interval $[e,w]$ is isomorphic to the Boolean algebra $\mathfrak{B}_{\ell(w)}$ of rank $\ell(w)$, where $\ell(w)$ is the length of~$w$. The \defi{Bruhat interval polytope~$Q_{v,w}$}, introduced in~\cite{ts-wi}, is the convex hull of the points $(u(1),\dots,u(n))$ in $\R^{n}$ for all $v\leq u\leq w$.
The Schubert variety $X_w$ is a smooth toric variety if and only if $Q_{e,w}$ is combinatorially equivalent to the $\ell(w)$-dimensional cube $I^{\ell(w)}$, see \cite{LMP1}.

\begin{theorem} \cite{Fan,Karup,Tenner2007,Tenner2012,LMP1} \label{thm:1}The following are equivalent:
\begin{enumerate}
\item[$(0)$] $X_w$ is a toric variety \textup{(}i.e., of complexity zero\textup{)}.
\item[$(1)$] $X_{w}$ is a smooth toric variety.
\item[$(2)$] $w$ avoids the patterns $321$ and $3412$.
\item[$(3)$] A reduced decomposition of $w$ consists of distinct letters.
\item[$(4)$] $X_w$ is isomorphic to a Bott--Samelson variety.
\item[$(5)$] The Bruhat interval $[e,w]$ is isomorphic to $\mathfrak{B}_{\ell(w)}$, the Boolean algebra of rank $\ell(w)$.
\item[$(6)$] The Bruhat interval polytope $Q_{e,w}$ is combinatorially equivalent to the $\ell(w)$-dimensional cube.
\end{enumerate}
\end{theorem}

Recall that for a decomposition $\underline{w}=s_{i_{1}}\dots s_{i_{\ell}}$ of a permutation $w$, the Bott--Samelson variety $Z_{\underline{w}}$ is defined by the orbit space of $\overline{Bs_{i_1}B}\times\cdots\times \overline{Bs_{i_\ell}B}$ by the right action of $B^{\ell}$ in~\eqref{eq_def_of_Theta_BS}. Then $Z_{\underline{w}}$ has an iterated $\C P^{1}$-bundle structure because $\overline{B s_{i_k} B}/B \cong \C P^{1}$. Note that $\overline{B s_{i_k} B}/B=X_{s_{i_{k}}}$ and $\overline{B s_{i_k} B}/B \cong \mathcal{F}\ell(\C^{2})=\C P^{1}$.
A Bott--Samelson variety $Z_{\underline{w}}$ is not a toric variety in general, but it is diffeomorphic to a toric variety, called a Bott manifold. We refer the reader to~\cite{Gr-Ka} for more details.

In this article, we study an analog  of the equivalent statements~(1)$\sim$(6) in Theorem~\ref{thm:1} for Schubert varieties of complexity one. 
Whereas every toric Schubert variety is smooth, not every Schubert variety of complexity one is smooth. For example, the Schubert varieties $X_{3214}$ and $X_{3412}$ are of complexity one, but $X_{3214}$ is smooth and $X_{3412}$ is singular.

Jantzen~\cite{Jant} generalized the notion of Bott--Samelson varieties. For an ordered tuple of permutations $(w_{1},\dots ,w_{r})$, the variety $Z_{(w_{1},\dots,w_{r})}$ is the orbit space of $\overline{Bw_{1}B}\times\cdots\times \overline{Bw_{r}B}$ by the right action of $B^{r}$ in~\eqref{eq:action}. 
Unfortunately, there was no name for the variety $Z_{(w_{1},\dots,w_{r})}$ in~\cite{Jant}, but now, it is called a \defi{generalized Bott--Samelson variety} (cf.~\cite{BrionKannan} and \cite{Perrin}).
When $X_{w_{i}}$ is a flag manifold for every $i=1,\dots,r$, the variety $Z_{(w_{1},\dots,w_{r})}$ is called a \defi{flag Bott--Samelson variety}, see~\cite{FLS}.

\begin{theorem}\label{thm:smooth} For a permutation~$w$ in $\mathfrak{S}_n$, the following are equivalent:
\begin{enumerate}
\item[$(1^{\prime})$] $X_w$ is smooth and of complexity one.
\item[$(2^{\prime})$] $w$ contains the pattern $321$ exactly once and avoids the pattern $3412$.
\item[$(3^{\prime})$] There exists a reduced decomposition of $w$ containing $s_{i}s_{i+1}s_{i}$ as a factor and no other repetitions.
\item[$(4^{\prime})$]  $X_w$ is isomorphic to a flag Bott--Samelson variety $Z_{(w_{1},\dots,w_{r})}$ such that $r=\ell(w)-2$, $w_{k}=s_{j}s_{j+1}s_{j}$ for some $1\leq k\leq r$, $w_{i}=s_{j_{i}}$ for $i\neq k$, and $j_{1},\dots,j_{k-1},j_{k+1},\dots,j_{r},j,j+1$ are pairwise distinct.
\item[$(5^{\prime})$] The Bruhat interval $[e,w]$ is isomorphic to $\mathfrak{S}_3\times \mathfrak{B}_{\ell(w)-3}$. 
\item[$(6^{\prime})$]  The Bruhat interval polytope $Q_{e,w}$ is combinatorially equivalent to the product of the hexagon and the cube $I^{\ell(w)-3}$.
\end{enumerate}
\end{theorem}

\begin{theorem}\label{thm:singular} For a permutation~$w$ in $\mathfrak{S}_n$, the following are equivalent:
\begin{enumerate}
\item[$(1^{\prime\prime})$] $X_w$ is singular and of complexity one.
\item[$(2^{\prime\prime})$]  $w$ contains the pattern $3412$ exactly once and avoids the pattern $321$.
\item[$(3^{\prime\prime})$] There exists a reduced decomposition of $w$ containing $s_{i+1}s_{i}s_{i+2}s_{i+1}$ as a factor and no other repetitions.
\item[$(4^{\prime\prime})$]  $X_w$ is isomorphic to a generalized Bott--Samelson variety $Z_{(w_{1},\dots,w_{r})}$ such that $r=\ell(w)-3$, $w_{k}=s_{j+1}s_{j}s_{j+2}s_{j+1}$ for some $1\leq k \leq r$, $w_{i}=s_{j_{i}}$ for $i\neq k$, and $j_{1},\dots,j_{k-1},j_{k+1},\dots,j_{r},j,j+1,j+2$ are pairwise distinct.
\item[$(5^{\prime\prime})$] The Bruhat interval $[e,w]$ is isomorphic to $[e,3412]\times \mathfrak{B}_{\ell(w)-4}$. 
\item[$(6^{\prime\prime})$]  The Bruhat interval polytope $Q_{e,w}$ is combinatorially equivalent to the product of $Q_{e,3412}$ and the cube $I^{\ell(w)-4}$.
\end{enumerate}
\end{theorem}

The equivalence between the first two statements (respectively, the second and the third statements) in Theorems~\ref{thm:smooth} and~\ref{thm:singular} is an immediate consequence of~\cite{L-S1990} and~\cite{Tenner2012} (respectively,~\cite{Daly}). The following diagram shows how we prove the main theorems in the paper. We prove Theorems~\ref{thm:smooth} and~\ref{thm:singular} in parallel.

\begin{figure}[ht]
\centering
 \tikz{
\node[draw=gray!50,fill=gray!20] (a) at (0,3) [circle] {($1^{\prime}$)};
\node[draw=gray!50,fill=gray!20] (b) at (4,3) [circle] {($2^{\prime}$)};
\node[draw=gray!50,fill=gray!20] (c) at (8,3) [circle] {($3^{\prime}$)};
\node[draw=gray!50,fill=gray!20] (d) at (8,1) [circle] {($4^{\prime}$)};
\node[draw=gray!50,fill=gray!20] (e) at (4,1) [circle] {($5^{\prime}$)};
\node[draw=gray!50,fill=gray!20] (f) at (0,1) [circle] {($6^{\prime}$)};
\draw[<->,thick] (a) -- (b) node[midway,above] {Proposition~\ref{prop:complexity-pattern}}
node[midway, below] {(\cite{L-S1990,Tenner2012})};
\draw [<->,thick](c) -- (b) node[midway,above] {Theorem~\ref{thm:Daly}}
node[midway, below] {(\cite{Daly})};
\draw[->,thick]  (a) -- (d) node[midway, above, sloped] {Theorem~4.7};
\draw[<-,thick] (e) -- (d) node[midway, below] {Corollary~\ref{cor_Xw_BS_iso}};
\draw[->,thick] (e) -- (a) node[midway, below, sloped] {Proposition~5.2};
\draw[<->,thick] (a) -- (f) node[midway, left] {Theorem~\ref{thm:poly-comb}};
}
\end{figure}

Like as a Bott--Samelson variety is diffeomorphic to a Bott manifold having a higher rank torus action, a flag Bott--Samelson variety  is diffeomorphic to a \textit{flag Bott manifold} which admits a higher rank torus action. Whereas a Bott manifold is a toric variety, a flag Bott manifold is not a toric variety in general, but it becomes a GKM manifold. In addition, we will see that every smooth Schubert variety of complexity one is diffeomorphic to a flag Bott manifold.

This paper is organized as follows. Section~\ref{sec:preliminaries} contains basic notions and facts about symmetric groups, Schubert varieties, Bott--Samelson varieties and Bott towers. In Section~\ref{sec:pattern-complexity}, we introduce the various relation between the pattern avoidance of a permutation and the complexity of a Schubert variety, and see the equivalence among the first three statements in Theorems~\ref{thm:smooth} and~\ref{thm:singular}. In Section~\ref{sec:flag-bott-samelson}, we introduce the notions of flag Bott--Samelson varieties and generalized Bott--Samelson varieties, and prove the implications $(1^{\prime})\Rightarrow (4^{\prime})\Rightarrow (5^{\prime})$ and $(1^{\prime\prime})\Rightarrow (4^{\prime\prime})\Rightarrow (5^{\prime\prime})$ in Theorems~\ref{thm:smooth} and~\ref{thm:singular}, respectively. In Section~\ref{sec:moment polytopes}, we study the properties of Bruhat intervals and Bruhat interval polytopes related to Schubert varieties of complexity one, and then  complete proofs of Theorems~\ref{thm:smooth} and~\ref{thm:singular}. In Section~\ref{sec:flag-bott}, we introduce the notion of flag Bott manifolds, and then show that every smooth Schubert variety of complexity one is diffeomorphic to a flag Bott manifold.

%%%%%%%%%%%%%%%%%%%%%%%%%%%%%%%%
\section{Preliminaries}\label{sec:preliminaries}

In this section, we first prepare basic facts about symmetric groups and Schubert varieties from~\cite{Brion2005}, and then see the relation among Schubert varieties, Bott--Samelson varieties, and Bott towers.

Let $G = \GL_n(\C)$ and $B$ the set of upper triangular matrices in $G$. We denote by $\T$ the set of diagonal matrices in $G$. Then, $\T \cong (\Cstar)^n$. 
The homogeneous space $G/B$ is a smooth projective variety can be identified with the set 
\[
\flag(\C^n) = \{ (\{0\} \subsetneq V_1 \subsetneq V_2 \subsetneq \cdots \subsetneq V_n = \C^n) \mid 
\dim_{\C} V_i = i \quad \text{ for }i =1,\dots,n\}
\]
of chains of subspaces of $\C^n$.
The Weyl group of $G$ is identified with the symmetric group $\mathfrak{S}_n$ on the set $[n] \coloneqq \{1,2,\dots,n\}$. For an element $w\in\mathfrak{S}_n$, we use the one-line notation 
\[
w=w(1)w(2)\cdots w(n).
\]
In this one-line notation, the identity element $e$ is presented by $e = 1 \ 2 \ \cdots \ n$.

We denote the set of transpositions by
\begin{equation}\label{eq:def_of_T}
T=\{(i,j)\mid 1\leq i < j\leq n\},
\end{equation}
which are permutations on $[n]$ swapping $i$ and $j$.
The \defi{simple transpositions} $s_i$ are the transpositions of the form
\[
s_i \coloneqq (i,i+1), \quad \text{ for }i=1,\ldots,n-1.
\]
Since $\mathfrak{S}_n$ is generated by simple transpositions, every $w\in\mathfrak{S}_n$ can be expressed as a product of simple transpositions. 
If $w=s_{i_1}\cdots s_{i_\ell}$ and $\ell$ is minimal among all such expressions, then $\ell$ is called the \defi{length} of~$w$ (written $\ell(w)=\ell$) and the expression $s_{i_1}\cdots s_{i_\ell}$ is called a \defi{reduced decomposition} (or \defi{reduced expression} or \defi{reduced word}) for~$w$. 
We denote by $R(w)$ the set consisting of all reduced decompositions of $w.$
A consecutive substring of a reduced decomposition is called a \defi{factor}. For instance, since
$$R(321)=\{s_{1}s_{2}s_{1},\,s_{2}s_{1}s_{2}\}\text{ and }R(3412)=\{s_{2}s_{3}s_{1}s_{2},\,s_{2}s_{1}s_{3}s_{2}\},$$ no reduced decomposition for $3412$ contains $s_{i}s_{i+1}s_{i}$ as a factor unlike $321$.

The Bruhat order on $\mathfrak{S}_n$ is defined by $v\leq w$  if a reduced decomposition of $v$ is a substring of some reduced decomposition of $w$. Then $\mathfrak{S}_n$ with the \defi{Bruhat order} is a graded poset, with rank function given by length. Figure~\ref{fig:bruhat_order} gives the Hasse diagram for the Bruhat order on $\mathfrak{S}_4$.

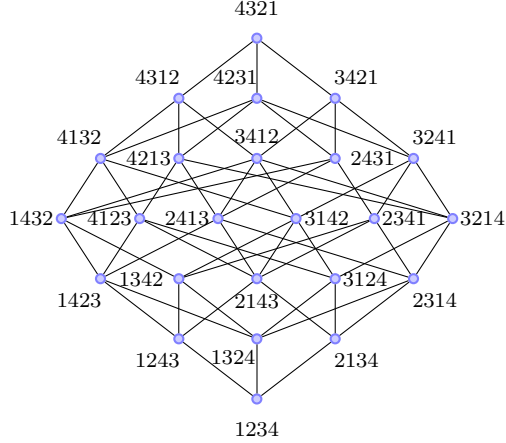
\begin{figure}[h!]
	\centering
	\begin{tikzpicture}[scale=.7]
	\tikzset{every node/.style={font=\footnotesize}}
	\matrix [matrix of math nodes,column sep={0.52cm,between origins},
	row sep={0.8cm,between origins},
	nodes={circle, draw=blue!50,fill=blue!20, thick, inner sep = 0pt , minimum size=1.2mm}]
	{
		& & & & & \node[label = {above:{4321}}] (4321) {} ; & & & & & \\
		& & &
		\node[label = {above left:4312}] (4312) {} ; & &
		\node[label = {above left:4231}] (4231) {} ; & &
		\node[label = {above right:3421}] (3421) {} ; & & & \\
		& \node[label = {above left:4132}] (4132) {} ; & &
		\node[label = {left:4213}] (4213) {} ; & &
		\node[label = {[label distance = -0.1cm] above:3412}] (3412) {} ; & &
		\node[label = {[label distance = 0.1cm]0:2431}] (2431) {} ; & &
		\node[label = {above right:3241}] (3241) {} ; & \\
		\node[label = {left:1432}] (1432) {} ; & &
		\node[label = {left:4123}] (4123) {} ; & &
		\node[label = {[label distance = 0.01cm]180:2413}] (2413) {} ; & &
		\node[label = {[label distance = 0.01cm]0:3142}] (3142) {} ; & &
		\node[label = {right:2341}] (2341) {} ; & &
		\node[label = {right:3214}] (3214) {} ; \\
		& \node[label = {below left:1423}] (1423) {} ; & &
		\node[label = {[label distance = 0.1cm]182:1342}] (1342) {} ; & &
		\node[label = {[label distance = -0.1cm] below:2143}] (2143) {} ; & &
		\node[label = {right:3124}] (3124) {} ; & &
		\node[label = {below right:2314}] (2314) {} ; & \\
		& & & \node[label = {below left:1243}] (1243) {} ; & &
		\node[label = {[label distance = 0.01cm]190:1324}] (1324) {} ; & &
		\node[label = {below right:2134}] (2134) {} ; & & & \\
		& & & & & \node[label = {below:1234}] (1234) {} ; & & & & & \\
	};		
	\draw (4321)--(4312)--(4132)--(1432)--(1423)--(1243)--(1234)--(2134)--(2314)--(2341)--(3241)--(3421)--(4321);
	\draw (4321)--(4231)--(4132);
	\draw (4231)--(3241);
	\draw (4231)--(2431);
	\draw (4231)--(4213);
	\draw (4312)--(4213)--(2413)--(2143)--(3142)--(3241);
	\draw (4312)--(3412)--(2413)--(1423)--(1324)--(1234);
	\draw (3421)--(3412)--(3214)--(3124)--(1324);
	\draw (3421)--(2431)--(2341)--(2143)--(2134);
	\draw (4132)--(4123)--(1423);
	\draw (4132)--(3142)--(3124)--(2134);
	\draw (4213)--(4123)--(2143)--(1243);
	\draw (4213)--(3214);
	\draw (3412)--(1432)--(1342)--(1243);
	\draw (2431)--(1432);
	\draw (2431)--(2413)--(2314);
	\draw (3142)--(1342)--(1324);
	\draw (4123)--(3124);
	\draw (2341)--(1342);
	\draw (2314)--(1324);
	\draw (3412)--(3142);
	\draw (3241)--(3214)--(2314);
	\end{tikzpicture}
	\caption{The Bruhat order on $\mathfrak{S}_4$.}\label{fig:bruhat_order}
\end{figure}

The complex torus $\T$ acts on $G/B$ by the left multiplication, and the set of $\T$-fixed points is identified with $\mathfrak{S}_n$.
More precisely, each element $w \in \mathfrak{S}_n$ corresponds to a coordinate flag given by
\[
(\{0\} \subsetneq \langle \mathbf{e}_{w(1)} \rangle \subsetneq \langle \mathbf{e}_{w(1)}, \mathbf{e}_{w(2)} \rangle \subsetneq \cdots \subsetneq V_n = \C^n)
\]
where $\mathbf{e}_1,\dots, \mathbf{e}_n$ are the standard basis vectors in $\C^n$. We denote by $wB$ this standard coordinate flag. It is well-known that 
 $\flag(\C^n)$ has a Bruhat decomposition 
$$\flag(\C^n)=\bigsqcup_{w\in \mathfrak{S}_n} BwB/B.$$
Moreover, the $B$-orbit $BwB/B$ is isomorphic to $\C^{\ell(w)}$ and called a \defi{Schubert cell}. The (Zariski) closure of $BwB/B$ is the \defi{Schubert variety} $X_w$, and  each Schubert variety decomposes into Schubert cells:
\begin{equation*}
X_w=\bigsqcup_{v\leq w} BvB/B.
\end{equation*}

Note that most Schubert varieties are singular, and they are desingularized using Bott--Samelson varieties. Let $w$ be a permutation and consider a decomposition $\underline{w}=({i_1},\dots, {i_\ell})$ of $w$ (not necessarily reduced). The \defi{Bott--Samelson variety} associated with $\underline{w}$, denoted $Z_{\underline{w}}$, is the quotient of $\overline{Bs_{i_1}B}\times\cdots\times \overline{Bs_{i_\ell}B}$ by the action of $B^\ell \coloneqq \underbrace{B \times \cdots \times B}_{\ell}$ given by:
\begin{equation}\label{eq_def_of_Theta_BS}
(b_1,\dots,b_\ell)\cdot (p_1,\dots,p_\ell)
	\coloneqq (p_1b_1, b_1^{-1}p_2b_2,\dots,b_{\ell-1}^{-1}p_\ell b_\ell)
\end{equation}
for $(b_1,\dots,b_{\ell}) \in B^{\ell}$ and $(p_1,\dots,p_{\ell}) \in \prod_{k=1}^{\ell} \overline{B s_{i_k} B}$.
Then $Z_{\underline{w}}$ is smooth and it has an iterated $\C P^1$-bundle structure because $\overline{B s_{i_k} B}/B \cong \C P^1$. 
Moreover, the left multiplication of $B$ on $\overline{B s_{i_1} B}$ induces an action of $B$ on $Z_{\underline{w}}$ and we have a 
$B$-equivariant map
$$p_{\underline{w}} \colon Z_{\underline{w}}\to G/B$$
defined by $p_{\underline{w}}(p_1,\dots,p_\ell)=p_1\cdots p_\ell B$. If $\underline{w}$ is a reduced decomposition of $w$, then $p_{\underline{w}}$ gives a resolution of singularities for $X_w$. See~\cite{bo-sa,De74} for details.

A Bott--Samelson variety $Z_{\underline{w}}$ is not a toric variety in general, but it is diffeomorphic to a toric variety, called a Bott manifold.
\begin{definition}\cite{Gr-Ka}\label{def:bott}
A \defi{Bott tower} is an iterated $\C P^1$-bundle:
\begin{equation*}
\begin{tikzcd}[row sep = 0.2cm]
\mathcal{B}_{\ell} \arrow[r, "\pi_{\ell}"] \arrow[d, equal]&  
\mathcal{B}_{\ell-1} \arrow[r, "\pi_{\ell-1}"] &
\cdots \arrow[r, "\pi_2"] &
\mathcal{B}_1 \arrow[r, "\pi_1"] \arrow[d, equal]&
\mathcal{B}_0, \arrow[d, equal]\\
P(\underline{\C}\oplus \xi_{\ell}) & & & \C P^1 & \{\text{a point}\}
\end{tikzcd} 
\end{equation*}
 where each $\mathcal{B}_{k}$ is the complex projectivization of the Whitney sum of a holomorphic line bundle $\xi_{k}$ over $\mathcal{B}_{k-1}$ and the trivial bundle $\underline{\C}$.
Each $\mathcal{B}_{k}$ is called a \defi{Bott manifold} (of height $k$).
\end{definition}
Let $\gamma_j$ be the tautological line bundle over $\mathcal{B}_j$ and $\gamma_{j,i}$ the pullback of $\gamma_j$ by the projection $\pi_i\circ \cdots \circ\pi_{j+1}\colon \mathcal{B}_{i}\to \mathcal{B}_j$ for $i>j$. For convenience, we define $\gamma_{j,j}=\gamma_j$. Then for each $k=2,\ldots,\ell$, there exist $a_{j,k} \in \Z$ for $1 \leq j< k$ such that
\[
\xi_{k}=\bigotimes_{1\leq j < k}\gamma_{j,k-1}^{\otimes a_{j,k}}
\] 
in Definition~\ref{def:bott} since the Picard number of $\mathcal{B}_k$ is $k$ (cf.~\cite[Exercise~II.7.9]{Hartshorne}). Each Bott tower is  determined by the list of the integers $a_{j,k}$ ($1\leq j< k\leq \ell$), and we visualize them as an upper triangular matrix
\begin{equation*}
\begin{pmatrix}
 0& a_{1,2} & a_{1,3} & \dots  & a_{1,\ell}\\
  &   0& a_{2,3}& \dots & a_{2,\ell} \\
  &  &  \ddots& \ddots & \vdots  \\
  &   &     & 0 &  a_{\ell-1,\ell} \\
  &   &     &   & 0   \\
\end{pmatrix}.
\end{equation*}
Note that each Bott manifold $\mathcal{B}_\ell$ is a  smooth projective toric variety whose fan is determined by the above matrix. The moment polytope of $\mathcal{B}_\ell$ is combinatorially equivalent to the $\ell$-dimensional cube $I^\ell$.

\begin{theorem}[{\cite[Proposition~3.10]{Gr-Ka}}]\label{thm:Gr-Ka}
For $w\in \mathfrak{S}_n$, let $\underline{w}=({i_1},{i_2},\dots, {i_\ell})$ be a reduced decomposition. Then
the Bott--Samelson variety $Z_{\underline{w}}$ is diffeomorphic to a Bott manifold $\mathcal{B}_\ell$ determined by the integers 
\begin{equation}\label{eq:bott-integers}
a_{j,k}=\langle \mathbf{e}_{i_j}-\mathbf{e}_{i_j+1}, \mathbf{e}_{i_k}-\mathbf{e}_{i_k+1} \rangle
\end{equation} for $1\leq j<k\leq \ell$. Here, $\mathbf e_1,\dots,\mathbf e_{n+1}$ are the standard basis vectors in $\R^{n+1}$ and $\langle \cdot, \cdot \rangle$ is the standard inner product in $\R^{n+1}$.
\end{theorem}

Note that most Schubert varieties are neither smooth nor toric. However toric Schubert varieties are smooth and they are Bott manifolds.
For a toric Schubert variety $X_w$, every reduced decomposition of $w$ consists of distinct letters, and hence for the associated Bott manifold the integers $a_{j,k}$ in \eqref{eq:bott-integers} are either $0$ or $-1$.

\section{Pattern avoidance and the complexity}\label{sec:pattern-complexity}
In this section, we define the complexity of a Schubert variety using the notion of complexity of a torus action, and see the relation between the complexity of a Schubert variety and patterns of a permutation. We also show the equivalence among the first three statements in Theorems~\ref{thm:smooth} and~\ref{thm:singular}.

Let $X$ be a smooth complex projective algebraic variety having an action of algebraic torus $\T=(\C^\ast)^n$. When the maximal $\T$-orbit in $X$ has codimension~$k$, we call the number $k$ the complexity of the action.

Every Schubert variety $X_w$ is a $\T$-invariant irreducible subvariety of $\flag(\C^n)$.  Note that the $\T$-fixed point set of $X_w$ is the set of coordinate flags $uB$ for $u\leq w$. That is, there is a bijection between $(X_w)^{\T}$ and the Bruhat interval $$[e,w]\coloneqq\{v\in\mathfrak{S}_n\mid v\leq w\}.\footnote{For $v,w$ in $\mathfrak{S}_{n}$ with $v<w$ in Bruhat order, the Bruhat interval $[v,w]$ is the subposet of $(\mathfrak{S}_{n},<)$ consisting of all the permutations $u$ with $v\leq u\leq w$, and the Bruhat interval $[e,w]$ is also known as the principal order ideal of $w$, see~\cite{Tenner2007}.}$$ Using the Pl\"{u}cker embedding, we get a moment map $\mu\colon \flag(\C^n) \to \R^n$ which sends 
\[
uB\mapsto (u^{-1}(1),\ldots,u^{-1}(n)),
\] 
and the image $\mu(\flag(\C^n))$ is a simple convex polytope
\[
\Perm_{n-1} \coloneqq \Conv \{(w(1),\dots,w(n)) \in \R^n \mid w \in \mathfrak{S}_n \},
\]
called the \defi{permutohedron}. 

The notion of Bruhat interval polytopes was introduced in~\cite{ts-wi} as a generalization of the notion of permutohedra. For two elements $v$ and $w$ in $\mathfrak{S}_n$ with $v\leq w$ in Bruhat order, the Bruhat interval polytope~$Q_{v,w}$ is the convex polytope  given by the convex hull of the points $(u(1),\dots,u(n))\in\R^n$ for $v\leq u\leq w$. Then
for a Schubert variety $X_w$, the moment map image
$\mu(X_w)$ becomes the Bruhat interval polytope 
\[
Q_{e,w^{-1}} \coloneqq \Conv\{(u(1),\ldots,u(n)) \in \R^n  \mid u\leq w^{-1} \},
\]
and moreover, the images $\mu(uB)$ for all $u\leq w$ are the vertices of the Bruhat interval polytope $Q_{e,w^{-1}}$. 
\begin{proposition}[{cf.~\cite[Proposition~6.20]{ts-wi} and references therein}]\label{vertices_of_BIP}
	The vertices of $Q_{e, w^{-1}}$ are the points $(u(1),\ldots,u(n)) \in \R^n$ for all $u\leq w^{-1}$.
\end{proposition}
We refer the readers to~\cite{LMP1} for more details on moment maps and the correspondence between Schubert varieties and Bruhat interval polytopes. 

\begin{remark}[{cf.~\cite[Remark~4.4]{LMP1}}]
For every $w\in \mathfrak{S}_{n}$ with $n\leq 4$, $Q_{v,w}$ and $Q_{v^{-1},w^{-1}}$ are combinatorially equivalent. However, 
for $w=35412$ in $\mathfrak{S}_{5}$, the Bruhat interval polytopes $Q_{e,w}$ and $Q_{e,w^{-1}}$ are not combinatorially equivalent. In fact, their $f$-vectors are different:
\begin{equation*}
f(Q_{e,35412})=(60,123,82,19,1)\text{ and } f(Q_{e,45132})=(60,122,81,19,1),
\end{equation*}
so $Q_{e,w}$ has one more edge than $Q_{e,w^{-1}}$. These vectors are computed by a computer program SageMath.
\end{remark}

Note that the complex dimension of a maximal $\mathbb{T}$-orbit in $X_w$ is the same as the real dimension of the moment map image $\mu(X_w)$. Hence we can define the \defi{complexity} of $X_w$ as follows:
$$c(w)=\dim_\C X_w - \dim_\R Q_{e,w^{-1}}=\ell(w)-\dim_\R Q_{e,w^{-1}}.$$
See~\cite[Section~6]{LMP1} for more details.
For example, $c(3142)=3-3=0$ and $c(4132)=4-3=1$, see Figure~\ref{fig:moment_map_images}.

\begin{figure}[hbt]
\begin{subfigure}[c]{0.49\textwidth}
	\begin{tikzpicture}[scale=6]
		\tikzset{every node/.style={draw=blue!50,fill=blue!20, circle, thick, inner sep=1pt,font=\footnotesize}}
\tikzset{red node/.style = {fill=red!20!white, draw=red!75!white}}
\tikzset{red line/.style = {line width=0.3ex, red, nearly opaque}}
	
	\coordinate (3142) at (1/3, 1/2, 1/6); 
	\coordinate (4231) at (2/3, 1/2, 1/6); 
	\coordinate (4312) at (5/6, 2/3, 1/2); 
	\coordinate (4321) at (5/6, 1/2, 1/3); 
	\coordinate (3421) at (5/6, 1/3, 1/2); 
	\coordinate (4213) at (2/3, 5/6, 1/2); 
	\coordinate (1324) at (1/3, 1/2, 5/6); 
	\coordinate (2413) at (2/3, 1/2, 5/6); 
	\coordinate (3412) at (5/6, 1/2, 2/3); 
	\coordinate (2314) at (1/2, 2/3, 5/6); 
	\coordinate (4123) at (1/2, 5/6, 1/3); 
	\coordinate (4132) at (1/2, 2/3, 1/6); 
	\coordinate (3214) at (1/2, 5/6, 2/3); 
	\coordinate (3124) at (1/3, 5/6, 1/2); 
	\coordinate (2431) at (2/3, 1/6, 1/2); 
	\coordinate (1432) at (1/2, 1/6, 2/3); 
	\coordinate (1423)  at (1/2, 1/3, 5/6); 
	\coordinate (1342)  at (1/3, 1/6, 1/2); 
	\coordinate (2341) at (1/2, 1/6, 1/3); 
	\coordinate (3241) at (1/2, 1/3, 1/6);
	\coordinate (1243) at (1/6, 1/3, 1/2); 
	\coordinate (2143) at (1/6, 1/2, 1/3);  
	\coordinate (1234) at (1/6, 1/2, 2/3); 
	\coordinate (2134) at (1/6, 2/3, 1/2); 
	
	\draw[thick, draw=blue!70] (4213)--(4312)--(3412)--(2413)--(2314)--(3214)--cycle;
	\draw[thick, draw=blue!70] (4312)--(4321)--(3421)--(3412);
	\draw[thick, draw=blue!70] (3421)--(2431)--(1432)--(1423)--(2413);
	\draw[thick, draw=blue!70] (1423)--(1324)--(2314);
	\draw[thick, draw=blue!70] (1432)--(1342)--(1243)--(1234)--(1324);
	\draw[thick, draw=blue!70] (1234)--(2134)--(3124)--(3214);
	\draw[thick, draw=blue!70] (3124)--(4123)--(4213);
	
	\draw[thick, draw=blue!70, dashed] (2134)--(2143)--(3142)--(4132)--(4123);
	\draw[thick, draw=blue!70, dashed] (2143)--(1243);
	\draw[thick, draw=blue!70, dashed] (3142)--(3241)--(2341)--(1342);
	\draw[thick, draw=blue!70, dashed] (2341)--(2431);
	\draw[thick, draw=blue!70, dashed] (3241)--(4231)--(4132);
	\draw[thick, draw=blue!70, dashed] (4231)--(4321);
	
	\draw[red line] (2134)--(1234)--(1324)--(2314)--cycle;
	\draw[red line] (1234)--(1243)--(1423)--(2413)--(2314);
	\draw[red line] (1324)--(1423);
	\draw[red line, dashed] (1243)--(2143)--(2413);
	\draw[red line, dashed] (2143)--(2134);
	
\node [label = {[label distance = 0cm]left:1234}, red node] at (1234) {};
\node[label = {[label distance = 0cm]left:1243}, red node] at (1243) {};
\node[label = {[label distance = 0cm]right:1324}, red node] at (1324) {};
\node[label = {[label distance = 0cm]left:1342}] at (1342) {};
\node [label = {[label distance = 0cm]above:1423},red node] at (1423) {};
\node[label = {[label distance = -0.2cm]below:1432}] at (1432) {};
\node [label = {[label distance = 0cm]left:2134}, red node] at (2134) {};
\node[label = {[label distance = -0.1cm]above right:2143}, red node] at (2143) {};
\node[label = {[label distance = 0cm]right:2314}, red node] at (2314) {};
\node[label = {[label distance = -0.1cm]below left:2341}] at (2341) {};
\node[label = {[label distance = -0.1cm]above :2413}, red node] at (2413) {};
\node[label = {[label distance = -0.2cm]below:2431}] at (2431) {};
\node[label = {[label distance = -0.2cm]above:3124}] at (3124) {};
\node[label = {[label distance = -0.2cm]above:3142}] at (3142) {};
\node[label = {[label distance = -0.2cm]above:3214}] at (3214) {};
\node [label = {[label distance = -0.1cm]above:3241}] at (3241) {};
\node[label = {[label distance = 0cm]below left:3412}] at (3412) {};
\node[label = {[label distance = 0cm]right:3421}] at (3421) {};
\node[label = {[label distance = -0.2cm]above:4123}] at (4123) {};
\node [label = {[label distance = 0cm]below:4132}] at (4132) {};
\node[label = {[label distance = 0cm]right:4213}] at (4213) {};
\node[label = {[label distance = 0cm]left:4231}] at (4231) {};
\node[label = {[label distance = 0cm]right:4312}] at (4312) {};
\node [label = {[label distance = 0cm]right:4321}] at (4321) {};

	\end{tikzpicture}
	\caption{$Q_{e, 3142^{-1}} = Q_{e, 2413}$.}
\end{subfigure}~\begin{subfigure}[c]{0.49\textwidth}
	\begin{tikzpicture}[scale=6]
		\tikzset{every node/.style={draw=blue!50,fill=blue!20, circle, thick, inner sep=1pt,font=\footnotesize}}
\tikzset{red node/.style = {fill=red!20!white, draw=red!75!white}}
\tikzset{red line/.style = {line width=0.3ex, red, nearly opaque}}
	
	\coordinate (3142) at (1/3, 1/2, 1/6); 
	\coordinate (4231) at (2/3, 1/2, 1/6); 
	\coordinate (4312) at (5/6, 2/3, 1/2); 
	\coordinate (4321) at (5/6, 1/2, 1/3); 
	\coordinate (3421) at (5/6, 1/3, 1/2); 
	\coordinate (4213) at (2/3, 5/6, 1/2); 
	\coordinate (1324) at (1/3, 1/2, 5/6); 
	\coordinate (2413) at (2/3, 1/2, 5/6); 
	\coordinate (3412) at (5/6, 1/2, 2/3); 
	\coordinate (2314) at (1/2, 2/3, 5/6); 
	\coordinate (4123) at (1/2, 5/6, 1/3); 
	\coordinate (4132) at (1/2, 2/3, 1/6); 
	\coordinate (3214) at (1/2, 5/6, 2/3); 
	\coordinate (3124) at (1/3, 5/6, 1/2); 
	\coordinate (2431) at (2/3, 1/6, 1/2); 
	\coordinate (1432) at (1/2, 1/6, 2/3); 
	\coordinate (1423)  at (1/2, 1/3, 5/6); 
	\coordinate (1342)  at (1/3, 1/6, 1/2); 
	\coordinate (2341) at (1/2, 1/6, 1/3); 
	\coordinate (3241) at (1/2, 1/3, 1/6);
	\coordinate (1243) at (1/6, 1/3, 1/2); 
	\coordinate (2143) at (1/6, 1/2, 1/3);  
	\coordinate (1234) at (1/6, 1/2, 2/3); 
	\coordinate (2134) at (1/6, 2/3, 1/2); 
	
	\draw[thick, draw=blue!70] (4213)--(4312)--(3412)--(2413)--(2314)--(3214)--cycle;
	\draw[thick, draw=blue!70] (4312)--(4321)--(3421)--(3412);
	\draw[thick, draw=blue!70] (3421)--(2431)--(1432)--(1423)--(2413);
	\draw[thick, draw=blue!70] (1423)--(1324)--(2314);
	\draw[thick, draw=blue!70] (1432)--(1342)--(1243)--(1234)--(1324);
	\draw[thick, draw=blue!70] (1234)--(2134)--(3124)--(3214);
	\draw[thick, draw=blue!70] (3124)--(4123)--(4213);
	
	\draw[thick, draw=blue!70, dashed] (2134)--(2143)--(3142)--(4132)--(4123);
	\draw[thick, draw=blue!70, dashed] (2143)--(1243);
	\draw[thick, draw=blue!70, dashed] (3142)--(3241)--(2341)--(1342);
	\draw[thick, draw=blue!70, dashed] (2341)--(2431);
	\draw[thick, draw=blue!70, dashed] (3241)--(4231)--(4132);
	\draw[thick, draw=blue!70, dashed] (4231)--(4321);
	
	\draw[red line] (2134)--(1234)--(1324)--(2314)--cycle;
	\draw[red line] (1234)--(1243)--(1342)--(1432)--(1423)--(1324);
	\draw[red line] (1432)--(2431)--(2413)--(2314);
	\draw[red line] (1423)--(2413);
	\draw[red line, dashed] (1243)--(2143)--(2134);
	\draw[red line, dashed] (2143)--(2341)--(2431);
	\draw[red line, dashed] (2341)--(1342);
	
\node [label = {[label distance = 0cm]left:1234}, red node] at (1234) {};
\node[label = {[label distance = 0cm]left:1243}, red node] at (1243) {};
\node[label = {[label distance = 0cm]right:1324}, red node] at (1324) {};
\node[label = {[label distance = 0cm]left:1342}, red node] at (1342) {};
\node [label = {[label distance = 0cm]above:1423},red node] at (1423) {};
\node[label = {[label distance = -0.2cm]below:1432}, red node] at (1432) {};
\node [label = {[label distance = 0cm]left:2134}, red node] at (2134) {};
\node[label = {[label distance = -0.1cm]above right:2143}, red node] at (2143) {};
\node[label = {[label distance = 0cm]right:2314}, red node] at (2314) {};
\node[label = {[label distance = -0.1cm]below left:2341}, red node] at (2341) {};
\node[label = {[label distance = -0.1cm]above :2413}, red node] at (2413) {};
\node[label = {[label distance = -0.2cm]below:2431}, red node] at (2431) {};
\node[label = {[label distance = -0.2cm]above:3124}] at (3124) {};
\node[label = {[label distance = -0.2cm]above:3142}] at (3142) {};
\node[label = {[label distance = -0.2cm]above:3214}] at (3214) {};
\node [label = {[label distance = -0.1cm]above:3241}] at (3241) {};
\node[label = {[label distance = 0cm]below left:3412}] at (3412) {};
\node[label = {[label distance = 0cm]right:3421}] at (3421) {};
\node[label = {[label distance = -0.2cm]above:4123}] at (4123) {};
\node [label = {[label distance = 0cm]below:4132}] at (4132) {};
\node[label = {[label distance = 0cm]right:4213}] at (4213) {};
\node[label = {[label distance = 0cm]left:4231}] at (4231) {};
\node[label = {[label distance = 0cm]right:4312}] at (4312) {};
\node [label = {[label distance = 0cm]right:4321}] at (4321) {};
	\end{tikzpicture}
	\caption{$Q_{e, 4132^{-1}} = Q_{e, 2431}$.}
\end{subfigure}
\caption{Moment map images of $X_{3142}$ and $X_{4132}$.}\label{fig:moment_map_images}
\end{figure}
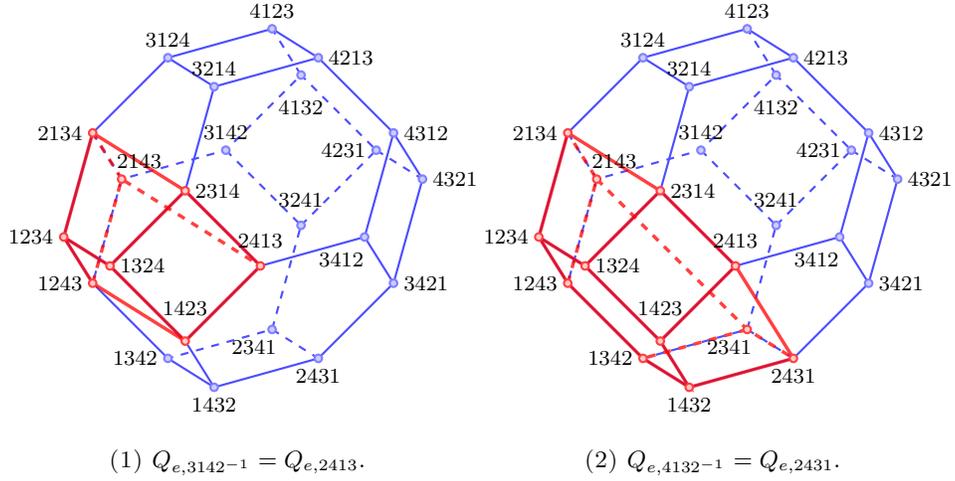

For a permutation $w\in \mathfrak{S}_n$, the dimension of the polytope $Q_{e,w}$ is related to a reduced decomposition of~$w$.
The \defi{support} of $w$ is the set  of distinct letters appearing in a reduced decomposition of $w$, and we denote it by $\supp(w)$. In fact, $\supp(w)$ is the same as the set of atoms in the Bruhat interval~$[e,w]$. 

It follows from~\cite[Corollary~7.13]{le-ma20} that the dimension of a Bruhat interval polytope~$Q_{e,w}$ is determined by the number of edges emanating from the vertex $(1,2,\dots,n)$, and moreover, by~\cite[Remark~7.5(5)]{le-ma20}, that number is the same as the cardinality of $\supp(w)$. Indeed, the vertices connecting to the vertex $(1,2,\dots,n)$ by an edge  are presented by 
\[
\begin{tikzcd}[row sep = 0cm, column sep = -0.2cm]
(1,2,\dots,i-1, & i+1,&i,&i+2,\dots,n)\\
&i\text{th}&&
\end{tikzcd}
\] for all $s_i\in \supp(w)$. Therefore, the complexity of $X_w$ is determined by the support of~$w$:
$$c(w)=\ell(w)-|\supp(w)|.$$
It means that $c(w)$ equals the number of repeated letters in a reduced decomposition of~$w$. In~\cite{Tenner2012}, Tenner denoted this quantity by~$\rep(w)$ and found a relation with patterns in $w$.

\begin{definition}
For $w\in\mathfrak{S}_n$ and $p\in\mathfrak{S}_k$ with~$k\leq n$, we say that the permutation~$w$ \defi{contains the pattern~$p$} if there exists a sequence $1\leq i_1<\cdots<i_k\leq n$ such that $w(i_1)\cdots w(i_k)$ is in the same relative order as $p(1)\cdots p(k)$. If $w$ does not contain~$p$, then we say that $w$ \defi{avoids~$p$}, or is \defi{$p$-avoiding}. 
\end{definition}

For example, the permutation $4231$ in $\mathfrak{S}_4$ has the pattern $321$ twice. Let $[321;3412](w)$ be the number of distinct $321$-and $3412$-patterns in a permutation~$w$. Then we can interpret Theorem~2.17 in~\cite{Tenner2012} in terms of the complexity of $X_w$.

\begin{theorem} \cite[Theorem~2.17]{Tenner2012}\label{Tenner-rep}
For a permutation $w$ in $\mathfrak{S}_n$, we have
\begin{enumerate}
\item $c(w)=0$ if and only if $[321;3412](w)=0$, and 
\item  $c(w)=1$ if and only if $[321;3412](w)=1$.
\end{enumerate}
\end{theorem}

\begin{example}\label{ex:complexity one-S4}
For $w\in \mathfrak{S}_4$, $[321;3412](w)=1$ if and only if $w$ is one of the following permutations
$$1432,3214,4132,4213,2431,3241,3412.$$
The first six permutations contains the $321$-pattern once, and the last one avoids it.
\end{example}

Recall that not every Schubert variety is smooth, and Lakshmibai and Sandhya characterized the smoothness of Schubert varieties in terms of pattern avoidance. 

\begin{theorem}\cite{L-S1990}\label{thm:LS}
 For a permutation $w\in\mathfrak{S}_n$, the Schubert variety $X_w$ is smooth if and only if $w$ avoids the patterns $3412$ and $4231$.
\end{theorem}

Combining Theorems~\ref{Tenner-rep} and \ref{thm:LS}, we obtain the following proposition. It shows the equivalence between the first two statements in Theorems~\ref{thm:smooth} and~\ref{thm:singular}.

\begin{proposition}\label{prop:complexity-pattern}
For $w\in\mathfrak{S}_n$, the following hold:
\begin{enumerate}
\item $X_w$ is a toric variety if and only if $w$ avoids the patterns both $321$ and $3412$;
\item $X_w$ is smooth and of complexity one if and only if $w$ has the pattern $321$ exactly once and avoids the pattern $3412$; and
\item $X_w$ is singular and of complexity one if and only if $w$ has the pattern $3412$ exactly once and avoids the pattern $321$.
\end{enumerate}
\end{proposition}

In general, $c(w)\leq [321;3412](w)$ and the equality holds only when $w$ avoids every pattern in the set
\begin{equation*}\label{eq:avoiding_pattern_set}
\{4321, 34512, 45123, 35412, 43512, 45132, 45213, 53412, 45312, 45231\},
\end{equation*} 
see \cite[Theorem 3.2]{Tenner2012}. For example, the permutation $4321$ has four distinct $321$-patterns but $c(4321)=3$. In the set above, every pattern except $4321$ has the pattern $3412$.
Since $w$ avoids the pattern $3412$ when $X_w$ is smooth, we get the following proposition.

\begin{proposition}\label{prop:complexity-pattern-general}
For a smooth Schubert variety $X_w$, the complexity of $X_w$ is less than or equal to the number of distinct $321$-patterns in $w$, and the equality holds if and only if $w$ avoids $4321$.
\end{proposition}

Since $c(w)$ equals the number of repeated letters in a reduced decomposition of~$w$, Theorem~\ref{Tenner-rep}(1) implies that a permutation $w\in \mathfrak{S}_n$ avoids both $321$-and $3412$-patterns if and only if every reduced decomposition of $w$ consists of distinct letters.  Daly characterized the reduced decomposition of the permutations satisfying $[321;3412](w)=1$ as follows.

\begin{theorem}\cite{Daly}\label{thm:Daly}
For a permutation $w\in\mathfrak{S}_n$, the following hold:
\begin{enumerate}
\item $w$ contains exactly one $321$ pattern and avoids $3412$ if and only if there exists a reduced decomposition of $w$ containing $s_{i}s_{i+1}s_{i}$ as a factor and no other repetitions. 
\item $w$ contains exactly one $3412$ pattern and avoids $321$ if and only if there exists a reduced decomposition of $w$ containing $s_{i+1}s_{i}s_{i+2}s_{i+1}$ as a factor and no other repetitions.
\end{enumerate}
\end{theorem}

The above theorem shows the equivalence between the second and the third statements in Theorems~\ref{thm:smooth} and~\ref{thm:singular}.

It is also shown in~\cite{Daly2010} that there is a one-to-one correspondence between $$\{w\in \mathfrak{S}_n\mid\text{ $w$ contains exactly one $321$ pattern and avoids $3412$}\}$$ and $$\{w\in\mathfrak{S}_{n+1}\mid \text{ $w$ contains exactly one $3412$ pattern and avoids $321$}\},$$ and the cardinality of the set is given in \href{https://oeis.org/A001871}{A001871} in OEIS~\cite{OEIS}. Thus we obtain that
\begin{align*}
&\#\{X_w\subseteq \flag(\C^n)\mid X_w\text{ is smooth and of complexity one}\}\\
&=\#\{X_{w'}\subseteq \flag(\C^{n+1}) \mid X_{w'}\text{ is singular and of complexity one}\}.
\end{align*}

%%%%%%%%%%%%%%%%%%%%%%%%%%%%%%%%%%%%%%%%%%%%

\section{Flag Bott--Samelson varieties and \\ generalized Bott--Samelson varieties}\label{sec:flag-bott-samelson}

In this section, we recall flag Bott--Samelson varieties from~\cite{FLS} and generalized Bott--Samelson varieties from~\cite{Jant}.

Let $G=\GL_n(\C)$. For a subset $I \subseteq [n] \coloneqq \{1,\dots,n\}$, we define the subgroup $W_I$ of $W$ by 
\[
W_I \coloneqq \langle s_i \mid i \in I \rangle.
\]
Then the \defi{parabolic subgroup} $P_I$ of $G$ corresponding to $I$ is defined by
\begin{equation*}
P_I =\overline{B w_I B} \subseteq G,
\end{equation*}
where $w_I$ is the longest element in $W_I$. 

\begin{definition}[{\cite[Definition~2.1]{FLS}}]
	Let $\mathcal I = (I_1,\dots,I_r)$ be a sequence of subsets of $[n]$. The \defi{flag Bott--Samelson variety} $Z_{\mathcal I}$ is defined by the orbit space
	\[
	Z_{\mathcal I} \coloneqq (P_{I_1} \times \cdots \times P_{I_r})/\Theta,
	\]
	where the right action $\Theta$ of $B^r \coloneqq \underbrace{B \times \cdots \times B}_{r}$ on $\prod_{k=1}^r P_{I_k}$ is defined by
	\begin{equation}\label{eq:action}
	\Theta((p_1,\dots,p_r), (b_1,\dots,b_r))
	= (p_1b_1, b_1^{-1}p_2b_2,\dots,b_{r-1}^{-1}p_rb_r)
	\end{equation}
	for $(p_1,\dots,p_r) \in \prod_{k=1}^r P_{I_k}$ and $(b_1,\dots,b_r)\in B^r$.
\end{definition}

Note that $P_{I_{k}}/B\cong\mathcal{F}\ell(\C^{|I_{k}|})$ for $k=1,\dots,r$. Hence if $|I_1|=\cdots=|I_r|=1$, then $Z_\mathcal{I}$ becomes a Bott--Samelson variety. Each Bott--Samelson variety can be described as an iterated $\C P^1$-bundle, whereas each flag Bott--Samelson variety can be described as an iterated bundle whose fiber at each stage is a flag manifold.

We recall properties of flag Bott--Samelson varieties from~\cite[Proposition~{2.3}]{FLS}.
The flag Bott--Samelson variety $Z_{\mathcal I}$ is a smooth projective variety and admits a nice decomposition of affine cells.
For $(w_1,\dots,w_r) \in \prod_{k=1}^r W_{I_k}$, we define $Z_{(w_1,\dots,w_r)}'$ and $\Zw{w_1,\dots,w_r}$ in $Z_{\mathcal I}$ by
\[
\begin{split}
Z_{(w_1,\dots,w_r)}' \coloneqq (Bw_1B \times \cdots \times B w_r B)/\Theta, \\
\Zw{w_1,\dots,w_r} \coloneqq (\overline{Bw_1B} \times \cdots \times \overline{B w_rB})/\Theta.
\end{split}
\]
We call $\Zw{w_1,\dots,w_r}$ a \defi{generalized Bott--Samelson variety}, see~\cite{Jant} and also~\cite{BrionKannan,Perrin}. Then $\Zwp{w_1,\dots,w_r}$ is an open dense subset of the generalized Bott--Samelson variety $\Zw{w_1,\dots,w_r}$ and we have that
\begin{equation*}
Z_{(w_1,\dots,w_r)}' \simeq \C^{\sum_{k=1}^r \ell(w_k)}.
\end{equation*}
For $w_{k}\in W_{I_{k}}$, since $\overline{Bw_{k}B} = \bigsqcup_{v \in W_I\atop v\leq w_{k}} BvB$, we have that
\begin{equation}\label{eq_Z_I}
Z_{(w_{1},\dots,w_{r})} = \bigsqcup_{(v_{1},\dots,v_{r}) \in \prod_{k=1}^r W_{I_k} \atop v_{k}\leq w_{k}\,\,\text{for } k=1,\dots,r} Z’_{(v_{1},\dots,v_{r} )}.
\end{equation}
Note that $Z_{\mathcal{I}}=Z_{(w_{I_{1}},\dots,w_{I_{r}})}$, where  $\mathcal{I}=(I_{1},\dots,I_{r})$ and $w_{I_{k}}$ is the longest element in $W_{I_{k}}$ for $k=1,\dots,r$.
\begin{example}\label{example_1213}
	Let $\mathcal I = (\{1,2\}, \{3\})$. Then we have that
\[
\begin{split}
Z_{\mathcal I} &= \Zwp{e,e} \sqcup \Zwp{s_1,e} \sqcup \Zwp{s_2,e} \sqcup \Zwp{s_1s_2,e} \sqcup \Zwp{s_2s_1,e} \sqcup \Zwp{s_1s_2s_1,e} \\
&\qquad \sqcup
 \Zwp{e,s_3} \sqcup \Zwp{s_1,s_3} \sqcup \Zwp{s_2,s_3} \sqcup \Zwp{s_1s_2,s_3} \sqcup \Zwp{s_2s_1,s_3} \sqcup \Zwp{s_1s_2s_1,s_3}.
\end{split}
\]
Each $\Zwp{w_1,w_2}$ is isomorphic to an affine cell as follows:
\begin{gather*}
\Zwp{e,e} \simeq \C^0, \quad 
\Zwp{s_1,e} \simeq \Zwp{s_2,e} \simeq  \Zwp{e,s_3} \simeq \C^1,\\
\Zwp{s_1s_2,e} \simeq \Zwp{s_2s_1,e} \simeq \Zwp{s_1,s_3} \simeq \Zwp{s_2,s_3} \simeq \C^2, \\
\Zwp{s_1s_2s_1,e} \simeq \Zwp{s_1s_2,s_3} \simeq \Zwp{s_2s_1,s_3} \simeq \C^3, \quad 
\Zwp{s_1s_2s_1,s_3} \simeq \C^4.
\end{gather*}
\end{example}

The multiplication map 
\[
p_{\mathcal I} \colon Z_{\mathcal I} \to G/B, \quad [p_1,\dots,p_r] \mapsto p_1 \cdots p_r B
\]
is well-defined because of the definition of the right action $\Theta$.
	The maximal torus~$\mathbb{T}$ acts on a generalized Bott--Samelson variety~$Z_{(w_{1},\dots,w_{r})}$ by
	$$t\cdot [g_{1},\dots,g_{r}]=[tg_{1},g_{2},\dots,g_{r}].$$
	We note that the multiplication map $p_{\mathcal I}$ is $\T$-equivariant.

\begin{proposition}[{\cite[Proposition~{2.7}]{FLS}}]\label{prop_multiplication_map_pI}
	Let $(w_1,\dots,w_r) \in \prod_{k=1}^r W_{I_k}$.
	Suppose that $w_1 \cdots w_r = v$ and $\ell(w_1) + \cdots + \ell(w_r) = \ell(v)$. Then the multiplication map $p_{\mathcal I}$ induces a birational morphism:
	\[
	p_{\mathcal I}|_{\Zw{w_1,\dots,w_r}} \colon \Zw{w_1,\dots,w_r} \to X_v.
	\]
	Indeed, we have an isomorphism between dense open subsets:
	\[
	\Zwp{w_1,\dots,w_r} \stackrel{\sim}{\longrightarrow} BvB/B. 
	\]
\end{proposition}
\begin{example}
	Let $G = \GL_3(\C)$, and let $\mathcal I = (\{1\}, \{2\}, \{1\})$. Then the flag Bott--Samelson variety $Z_{\mathcal I}$ has the decomposition:
	\[
	\begin{split}
	Z_{\mathcal I} &= \Zwp{e,e,e} \sqcup \Zwp{s_1,e,e} \sqcup \Zwp {e,s_2,e} \sqcup \Zwp{s_1,s_2,e}  \\
	& \qquad \sqcup \Zwp{e,e,s_1} \sqcup \Zwp{s_1,e,s_1} \sqcup \Zwp {e,s_2,s_1} \sqcup \Zwp{s_1,s_2,s_1}.
	\end{split}
	\]
	By Proposition~\ref{prop_multiplication_map_pI}, the multiplication map $p_{\mathcal I}$ induces the isomorphism $\Zwp{w_1,w_2,w_3} \cong B w_1w_2w_3 B/B$ except for $(w_1,w_2,w_3) = (s_1,e,s_1)$. Moreover, one can see that the multiplication map $p_{\mathcal I}$ is injective on $Z_{\mathcal I} \setminus (\Zwp{s_1,e,s_1}\sqcup \Zwp{s_1,e,e} \sqcup \Zwp{e,e,s_1})$. 
\end{example}
\begin{corollary}\label{cor_fBS_and_Schubert_isomorphic}
	Suppose that $I_1,\dots,I_r$ are pairwise disjoint subsets of $[n]$. 
	Then the multiplication map $p_{\mathcal I}$ with $\mathcal{I}=(I_1,\dots,I_r)$ induces an  isomorphism between $Z_{(w_{1},\dots,w_{r})}$ and its image $X_{w_1 \cdots w_r}$ as $\T$-varieties, where $w_k$ is an element in $W_{I_k}$. 
\end{corollary}
\begin{proof}
	We first note that $ X_{w_1 \cdots w_r} = \bigsqcup_{v \leq w_1 \cdots w_r} BvB/B$.
	Since $I_1,\dots,I_r$ are pairwise disjoint and $w_k \in W_{I_k}$, we have that
	\begin{equation*}
	v \leq w_1 \cdots w_r \iff v = v_1 \cdots v_r, \quad \text{ and }v_k \leq w_{k} \text{ in } W_{I_k}.
	\end{equation*}
	Therefore, we get that
	\begin{equation}\label{eq_X_w1_wr}
	X_{w_1 \cdots w_r} = \bigsqcup_{(v_1,\dots,v_r) \in \prod_{k=1}^r W_{I_k} \atop v_{k}\leq w_{k}\,\,\text{ for }k=1,\dots,r} B v_1 \cdots v_r B/B.
	\end{equation}
	This implies that the Schubert variety $X_{w_1 \cdots w_r}$ can be decomposed into affine cells indexed by elements in $\prod_{k=1}^r [e,w_{k}]$, where $[e,w_{k}]$ is a Bruhat interval in $W_{I_{k}}$. 
	Moreover, since $I_1,\dots,I_r$ are pairwise disjoint, we know that $\ell(v_1 \cdots v_r) = \ell(v_1) + \cdots + \ell(v_r)$. Hence by Proposition~\ref{prop_multiplication_map_pI}, the multiplication map~$p_{I}$ induces a $\T$-equivariant isomorphism
	\[
	\Zwp{v_1,\dots,v_r} \cong B v_1 \cdots v_r B/B
	\]
	for each $(v_1,\dots,v_r) \in \prod_{k=1}^r W_{I_k}$. Hence by combining the above isomorphism with equations~\eqref{eq_Z_I} and~\eqref{eq_X_w1_wr}, the result follows.
\end{proof}
\begin{example}\label{ex2}
	Continuing Example~\ref{example_1213}, the multiplication map $p_{\mathcal I}$ induces the  isomorphisms:
\[
\begin{array}{ll}
	\Zwp{e,e} \cong BeB/B, & \Zwp{s_1,e} \cong Bs_1B/B, \\ \Zwp{s_2,e} \cong B s_2 B/B, &
	\Zwp{s_1s_2,e} \cong B s_1s_2 B/B, \\ \Zwp{s_2s_1,e} \cong B s_2s_1 B/B, & \Zwp{s_1s_2s_1,e} \cong B s_1s_2s_1 B/B, \\
		\Zwp{e,s_3} \cong Bs_3B/B, & \Zwp{s_1,s_3} \cong Bs_1s_3B/B, \\ \Zwp{s_2,s_3} \cong B s_2s_3 B/B, &
	\Zwp{s_1s_2,s_3} \cong B s_1s_2s_3 B/B, \\ \Zwp{s_2s_1,s_3} \cong B s_2s_1s_3 B/B, & \Zwp{s_1s_2s_1,s_3} \cong B s_1s_2s_1s_3B/B 
	\end{array}
\]
Since we have the decomposition $ X_{w} = \bigsqcup_{v \le w} BvB/B$, we get an isomorphism $Z_{\mathcal I} = \Zw{s_1s_2s_1,s_3} \cong X_{s_1s_2s_1s_3}$. Indeed, $\{1,2\} \cap \{3\} = \emptyset$. 
\end{example}

\begin{theorem}\label{thm:main1}
Every Schubert variety $X_{w}$ of complexity one is $\T$-equivariantly isomorphic to a generalized Bott--Samelson variety. If $X_{w}$ is of complexity one and smooth, then it is $\T$-equivariantly isomorphic to a flag Bott--Samelson variety.
\end{theorem}
\begin{proof}
	By Proposition~\ref{prop:complexity-pattern} and Theorem~\ref{thm:Daly}, if $X_{w}$ is smooth, then there exists a reduced decomposition $\underline{w}=s_{i_1} \cdots s_{i_{\ell}}$ of $w$ containing $s_{i}s_{i+1} s_{i}$ as a factor and no other repetitions. On the other hand, if $X_{w}$ is singular, then there exists a reduced decomposition $\underline{w}=s_{i_1}\cdots s_{i_{\ell}}$ of $w$ containing $s_{i+1}s_{i}s_{i+2} s_{i+1}$ as a factor and no other repetitions.
	Hence there is a reduced decomposition $\underline{w}$ such that 
	\begin{enumerate}
	\item if $X_{w}$ is smooth, then $(i_{q}, i_{q+1}, i_{q+2}) = (i, i+1, i)$ for some $1 \le q \le \ell-2$; and
	\item if $X_{w}$ is singular, then $(i_{q}, i_{q+1}, i_{q+2}, i_{q+3}) = (i+1,i,i+2, i+1)$ for some $1 \le q \le \ell-3$.
	\end{enumerate}
	Now we define a sequence $\mathcal I = (I_1,\dots,I_{r})$ as follows:
	\begin{enumerate}
	\item If $X_{w}$ is smooth, then $r=\ell-2$ and we set
	\[
	I_k = \begin{cases}
	\{i_k\} & \text{ if } 1 \leq k < q, \\
	\{i, i+1\} & \text{ if } k = q, \\
	\{i_{k+2}\} & \text{ if } k > q.
	\end{cases}
	\]
	\item If $X_{w}$ is singular, then $r=\ell-3$ and we set
	\[
	I_k = \begin{cases}
	\{i_k\} & \text{ if } 1 \leq k < q, \\
	\{i, i+1, i+2\} & \text{ if } k = q, \\
	\{i_{k+3}\} & \text{ if } k > q.
	\end{cases}
	\]
	\end{enumerate}
	Then, in any case, the subsets $I_1,\dots,I_{r}$ are pairwise disjoint. Moreover, if $X_{w}$ is smooth, then the concatenation of longest elements in $W_{I_k}$ is the same as~${w}$. Hence by Corollary~\ref{cor_fBS_and_Schubert_isomorphic}, the Schubert variety $X_w$ is $\T$-equivariantly isomorphic to the flag Bott--Samelson variety $Z_{\mathcal I}$.
	
	If $X_{w}$ is singular, then we set 
	\[
	w_k = \begin{cases}
	s_{i_k} & \text{ if } 1 \leq k < q, \\
	s_{i+1}s_{i}s_{i+2}s_{i+1} & \text{ if } k = q, \\
	s_{i_{k}+3} & \text{ if } q<k\leq r.
	\end{cases}
	\]
	 Then $(w_{1},\dots,w_{r})\in \prod_{k=1}^{r} W_{I_{k}}$ and $w=w_{1}\dots w_{r}.$ Hence by Corollary~\ref{cor_fBS_and_Schubert_isomorphic}, $X_{w}$ is $\T$-equivariantly isomorphic to the generalized Bott--Samelson variety $Z_{(w_{1},\dots,w_{r})}$.
\end{proof}

The above theorem proves the implications $(1^{\prime})\Rightarrow (4^{\prime})$ and $(1^{\prime\prime})\Rightarrow (4^{\prime\prime})$ in Theorems~\ref{thm:smooth} and~\ref{thm:singular}, respectively. 

The following proposition can be checked easily and we omit the proof.

\begin{proposition}
There is a natural bijection between the set of $\mathbb{T}$-fixed points in $Z_{(w_1,\dots,w_r)}$ and the product of Bruhat intervals
$$[e,w_{1}]\times\dots\times [e,w_{r}].$$
\end{proposition}

If a Schubert variety $X_{w}$ and a generalized Bott--Samelson variety~$Z_{(w_{1},\dots,w_{r})}$ are isomorphic as $\T$-varieties, then they have the same set of $\mathbb{T}$-fixed points. Hence we get the following.

\begin{corollary}\label{cor_Xw_BS_iso}
If a Schubert variety $X_{w}$ is isomorphic to a generalized Bott--Samelson variety~$Z_{(w_{1},\dots,w_{r})}$ as $\mathbb{T}$-varieties, then the Bruhat interval $[e,w]$ is isomorphic to the product of Bruhat intervals $\prod_{k=1}^{r}[e,w_{k}]$.
\end{corollary}

The above corollary shows the implications $(4^{\prime})\Rightarrow (5^{\prime})$ and $(4^{\prime\prime})\Rightarrow (5^{\prime\prime})$ in Theorems~\ref{thm:smooth} and~\ref{thm:singular}, respectively. 

\section{Bruhat intervals and Bruhat interval polytopes}\label{sec:moment polytopes}

In this section, we study the properties of the Bruhat interval~$[e,w]$ and the Bruhat interval polytope~$Q_{e,w}$ for a permutation~$w$ with $c(w)\leq 1$. We will complete proofs of Theorems~\ref{thm:smooth} and~\ref{thm:singular}.

Recall that the following seven permutations in $\mathfrak{S}_4$ 
$$1432,3214,4132,4213,2431,3241,3412$$
satisfy $[321;3412](w)=1$ (in Example~\ref{ex:complexity one-S4}).
One can easily check that 
\begin{itemize}
\item If $w$ is either $1432$ or $3214$, then the Bruhat interval $[e,w]$ is isomorphic to $\mathfrak{S}_3$ and the Bruhat interval polytope $Q_{e,w}$ is combinatorially equivalent to the hexagon $\Perm_2$.
\item If $w$ is one of the permutations $4132,4213,2431$ and $3241$, then the Bruhat interval $[e,w]$ is isomorphic to $\mathfrak{S}_3\times \mathfrak{B}_1$ and the Bruhat interval polytope $Q_{e,w}$ is combinatorially equivalent to the hexagonal prism, $\Perm_2\times I$.
\item If $w=3412$, then the Bruhat interval $[e,w]$ and $Q_{e,w}$ are given in Figure~\ref{fig:BIP-3412}. Thus $[e,w]$ is isomorphic to neither $\mathfrak{S}_3\times \mathfrak{B}_1$ nor $\mathfrak{B}_{4}$, and $Q_{e,w}$ is combinatorially equivalent to neither the hexagonal prism nor a $4$-cube.
\end{itemize}
Here, $\mathfrak{B}_{\ell}$ denotes the Boolean algebra of length~$\ell$. The above situation happens in general and we have the following lemma that is obvious but plays a role in our argument.

\begin{figure}[b]
	\begin{subfigure}[c]{0.49\textwidth}
		\begin{tikzpicture}
		\tikzset{every node/.style={font=\footnotesize}}
			\tikzset{red node/.style = {fill=red!20!white, draw=red!75!white}}
	\tikzset{red line/.style = {line width=1ex, red,nearly transparent}}	
		\matrix [matrix of math nodes,column sep={0.48cm,between origins},
		row sep={1cm,between origins},
		nodes={circle, draw=blue!50,fill=blue!20, thick, inner sep = 0pt , minimum size=1.2mm}]
		{
			& & & & & \node[label = {above:{4321}}] (4321) {} ; & & & & & \\
			& & & 
			\node[label = {above left:4312}] (4312) {} ; & & 
			\node[label = {above left:4231}] (4231) {} ; & & 
			\node[label = {above right:3421}] (3421) {} ; & & & \\
			& \node[label = {above left:4132}] (4132) {} ; & & 
			\node[label = {left:4213}] (4213) {} ; & & 
			\node[label = {above:3412}, red node] (3412) {} ; & & 
			\node[label = {[label distance = 0.1cm]0:2431}] (2431) {} ; & & 
			\node[label = {above right:3241}] (3241) {} ; & \\
			\node[label = {left:1432}, red node] (1432) {} ; & & 
			\node[label = {left:4123}] (4123) {} ; & & 
			\node[label = {[label distance = 0.01cm]180:2413}, red node] (2413) {} ; & & 
			\node[label = {[label distance = 0.01cm]0:3142}, red node] (3142) {} ; & & 
			\node[label = {right:2341}] (2341) {} ; & & 
			\node[label = {right:3214}, red node] (3214) {} ; \\
			& \node[label = {below left:1423}, red node] (1423) {} ; & & 
			\node[label = {[label distance = 0.1cm]182:1342}, red node] (1342) {} ; & & 
			\node[label = {below:2143}, red node] (2143) {} ; & & 
			\node[label = {right:3124}, red node] (3124) {} ; & & 
			\node[label = {below right:2314}, red node] (2314) {} ; & \\
			& & & \node[label = {below left:1243}, red node] (1243) {} ; & & 
			\node[label = {[label distance = 0.01cm]190:1324}, red node] (1324) {} ; & & 
			\node[label = {below right:2134}, red node] (2134) {} ; & & & \\
			& & & & & \node[label = {below:1234}, red node] (1234) {} ; & & & & & \\
		};
		
		\draw (4321)--(4312)--(4132)--(1432)--(1423)--(1243)--(1234)--(2134)--(2314)--(2341)--(3241)--(3421)--(4321);
		\draw (4321)--(4231)--(4132);
		\draw (4231)--(3241);
		\draw (4231)--(2431);
		\draw (4231)--(4213);
		\draw (4312)--(4213)--(2413)--(2143)--(3142)--(3241);
		\draw (4312)--(3412)--(2413)--(1423)--(1324)--(1234);
		\draw (3421)--(3412)--(3214)--(3124)--(1324);
		\draw (3421)--(2431)--(2341)--(2143)--(2134);
		\draw (4132)--(4123)--(1423);
		\draw (4132)--(3142)--(3124)--(2134);
		\draw (4213)--(4123)--(2143)--(1243);
		\draw (4213)--(3214);
		\draw (3412)--(1432)--(1342)--(1243);
		\draw (2431)--(1432);
		\draw (2431)--(2413)--(2314);
		\draw (3142)--(1342)--(1324);
		\draw (4123)--(3124);
		\draw (2341)--(1342);
		\draw (2314)--(1324);
		\draw (3412)--(3142);
		\draw (3241)--(3214)--(2314);
		
		\draw[red line] (1234)--(1243)--(2143)--(2413);
		\draw[red line] (1234)--(2134)--(2143)--(3142);
		\draw[red line] (1243)--(1423);
		\draw[red line] (1243)--(1342);
		\draw[red line] (1234)--(1324);
		\draw[red line] (3124)--(2134)--(2314);
		\draw[red line] (1324)--(1423)--(1432)--(3412);
		\draw[red line] (1324)--(2314)--(3214)--(3412);
		\draw[red line] (1324)--(1342)--(1432);
		\draw[red line] (2314)--(2413)--(3412);
		\draw[red line] (1423)--(2413);
		\draw[red line] (1342)--(3142)--(3412);
		\draw[red line] (1324)--(3124)--(3214);
		\draw[red line] (3124)--(3142);
		\end{tikzpicture}
	\end{subfigure}
	\begin{subfigure}[c]{0.49\textwidth}
	\begin{tikzpicture}[scale=6]
		\tikzset{every node/.style={draw=blue!50,fill=blue!20, circle, thick, inner sep=1pt,font=\footnotesize}}
\tikzset{red node/.style = {fill=red!20!white, draw=red!75!white}}
\tikzset{red line/.style = {line width=0.3ex, red, nearly opaque}}	
	\coordinate (3142) at (1/3, 1/2, 1/6); 
	\coordinate (4231) at (2/3, 1/2, 1/6); 
	\coordinate (4312) at (5/6, 2/3, 1/2); 
	\coordinate (4321) at (5/6, 1/2, 1/3); 
	\coordinate (3421) at (5/6, 1/3, 1/2); 
	\coordinate (4213) at (2/3, 5/6, 1/2); 
	\coordinate (1324) at (1/3, 1/2, 5/6); 
	\coordinate (2413) at (2/3, 1/2, 5/6); 
	\coordinate (3412) at (5/6, 1/2, 2/3); 
	\coordinate (2314) at (1/2, 2/3, 5/6); 
	\coordinate (4123) at (1/2, 5/6, 1/3); 
	\coordinate (4132) at (1/2, 2/3, 1/6); 
	\coordinate (3214) at (1/2, 5/6, 2/3); 
	\coordinate (3124) at (1/3, 5/6, 1/2); 
	\coordinate (2431) at (2/3, 1/6, 1/2); 
	\coordinate (1432) at (1/2, 1/6, 2/3); 
	\coordinate (1423)  at (1/2, 1/3, 5/6); 
	\coordinate (1342)  at (1/3, 1/6, 1/2); 
	\coordinate (2341) at (1/2, 1/6, 1/3); 
	\coordinate (3241) at (1/2, 1/3, 1/6);
	\coordinate (1243) at (1/6, 1/3, 1/2); 
	\coordinate (2143) at (1/6, 1/2, 1/3);  
	\coordinate (1234) at (1/6, 1/2, 2/3); 
	\coordinate (2134) at (1/6, 2/3, 1/2); 
	\draw[thick, draw=blue!70] (4213)--(4312)--(3412)--(2413)--(2314)--(3214)--cycle;
	\draw[thick, draw=blue!70] (4312)--(4321)--(3421)--(3412);
	\draw[thick, draw=blue!70] (3421)--(2431)--(1432)--(1423)--(2413);
	\draw[thick, draw=blue!70] (1423)--(1324)--(2314);
	\draw[thick, draw=blue!70] (1432)--(1342)--(1243)--(1234)--(1324);
	\draw[thick, draw=blue!70] (1234)--(2134)--(3124)--(3214);
	\draw[thick, draw=blue!70] (3124)--(4123)--(4213);	
	\draw[thick, draw=blue!70, dashed] (2134)--(2143)--(3142)--(4132)--(4123);
	\draw[thick, draw=blue!70, dashed] (2143)--(1243);
	\draw[thick, draw=blue!70, dashed] (3142)--(3241)--(2341)--(1342);
	\draw[thick, draw=blue!70, dashed] (2341)--(2431);
	\draw[thick, draw=blue!70, dashed] (3241)--(4231)--(4132);
	\draw[thick, draw=blue!70, dashed] (4231)--(4321);	
\node [label = {[label distance = 0cm]left:1234}, red node] at (1234) {};
\node[label = {[label distance = 0cm]left:1243}, red node] at (1243) {};
\node[label = {[label distance = 0cm]right:1324}, red node] at (1324) {};
\node[label = {[label distance = 0cm]left:1342}, red node] at (1342) {};
\node [label = {[label distance = 0cm]above:1423},red node] at (1423) {};
\node[label = {[label distance = -0.2cm]below:1432}, red node] at (1432) {};
\node [label = {[label distance = 0cm]left:2134}, red node] at (2134) {};
\node[label = {[label distance = -0.1cm]below right:2143}, red node] at (2143) {};
\node[label = {[label distance = 0cm]below:2314}, red node] at (2314) {};
\node[label = {[label distance = 0cm]right:2341}] at (2341) {};
\node[label = {[label distance = 0cm]left:2413}, red node] at (2413) {};
\node[label = {[label distance = -0.2cm]below:2431}] at (2431) {};
\node[label = {[label distance = -0.2cm]above:3124}, red node] at (3124) {};
\node[label = {[label distance = -0.2cm]above:3142}, red node] at (3142) {};
\node[label = {[label distance = -0.2cm]above:3214}, red node] at (3214) {};
\node [label = {[label distance = -0.1cm]above:3241}] at (3241) {};
\node[label = {[label distance = 0cm]below left:3412}, red node] at (3412) {};
\node[label = {[label distance = 0cm]right:3421}] at (3421) {};
\node[label = {[label distance = -0.2cm]above:4123}] at (4123) {};
\node [label = {[label distance = 0cm]below:4132}] at (4132) {};
\node[label = {[label distance = 0cm]right:4213}] at (4213) {};
\node[label = {[label distance = 0cm]left:4231}] at (4231) {};
\node[label = {[label distance = 0cm]right:4312}] at (4312) {};
\node [label = {[label distance = 0cm]right:4321}] at (4321) {};
	\draw[red line] (3214)--(2314)--(2413)--(3412)--cycle;
	\draw[red line] (2314)--(1324)--(1423)--(2413);
	\draw[red line] (1234)--(1243)--(1342)--(1432)--(1423);
	\draw[red line] (1432)--(3412);
	\draw[red line] (1324)--(1234)--(2134)--(3124);
	\draw[red line](3124)--(3214);
	\draw[red line, dashed] (3124)--(3142)--(2143)--(1243);
	\draw[red line, dashed] (3142)--(3412);
	\draw[red line, dashed] (2134)--(2143);
	\draw[red line, dashed] (3142)--(1342);
	\end{tikzpicture}
\end{subfigure}
\caption{The Bruhat interval $[e,3412]$ and the Bruhat interval polytope $Q_{e,3412}$.}\label{fig:BIP-3412}
\end{figure}
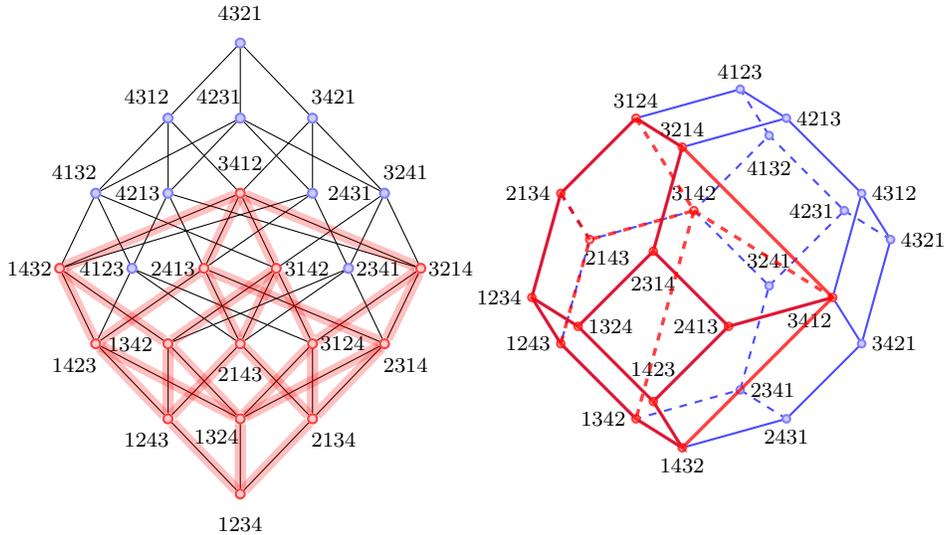

\begin{lemma}\label{lem:repeated-elements}
For any positive integer $i$, the following hold:
\begin{enumerate}
\item if $w=s_is_{i+1}s_i$, then $[e,w]$ is isomorphic to $\mathfrak{S}_3$ as a poset, and $Q_{e,w}$ is combinatorially equivalent to the hexagon $\Perm_2$; and
\item if $w=s_{i+1}s_is_{i+2}s_{i+1}$, then $[e,w]$ is isomorphic to $[e,3412]$ as a poset, and $Q_{e,w}$ is combinatorially equivalent to $Q_{e,3412}$.
\end{enumerate}
\end{lemma}
Therefore, if $w=s_{i+1}s_is_{i+2}s_{i+1}$ for some positive integer~$i$, then $[e,w]$ is
isomorphic to neither $\mathfrak{S}_3\times \mathfrak{B}_1$ nor $\mathfrak{B}_{4}$, and $Q_{e,w}$ is combinatorially equivalent to neither the hexagonal prism nor a $4$-cube.

\begin{proposition}\label{thm:bruhat-interval}
For a permutation $w$ in $\mathfrak{S}_n$, the following hold:
\begin{enumerate}
\item If the Bruhat interval $[e,w]$ is isomorphic to $\mathfrak{S}_3\times \mathfrak{B}_{\ell(w)-3}$ \textup{(}as a poset\textup{)}, then $X_{w}$ is smooth and of complexity one.
\item If the Bruhat interval $[e,w]$ is isomorphic to $[e,3412]\times\mathfrak{B}_{\ell(w)-4}$ \textup{(}as a poset\textup{)}, then $X_{w}$ is singular and of complexity one.
\end{enumerate}
\end{proposition}
Note that the converse of each statement in the above proposition is also true by Theroem~\ref{thm:main1} and Corollary~\ref{cor_Xw_BS_iso}.
\begin{proof}[Proof of Proposition~\ref{thm:bruhat-interval}]
Let us prove the first statement.
Assume that the Bruhat interval $[e,w]$ is isomorphic to a poset $\mathfrak{S}_3\times \mathfrak{B}_{\ell-3}$. Note that
\begin{enumerate}
\item $\mathfrak{S}_3\times\mathfrak{B}_{\ell-3}$ has $\ell-1$ atoms;
\item $\mathfrak{S}_3\times\mathfrak{B}_{\ell-3}$ has rank $\ell$; and
\item every subinterval of $\mathfrak{S}_3\times\mathfrak{B}_{\ell-3}$ is isomorphic to $\mathfrak{S}_3\times\mathfrak{B}_{k}$ ($k<\ell-3$) or $\mathfrak{B}_{k'}$ ($k'\leq \ell-1$).
\end{enumerate}
From (1) and (2) above, we get $c(w)=1$. If $X_w$ is not smooth, then $w$ contains $3412$ exactly once and avoids $321$. Then by Theorem~\ref{thm:Daly}, there is a reduced decomposition $\underline{w}$ of $w$ such that
\begin{equation*}
\underline{w}=s_{i_1}\cdots s_{i_{k-1}} s_{i_{k+1}} s_{i_k} s_{i_{k+2}} s_{i_{k+1}} s_{i_{k+3}}\cdots s_{i_{\ell-1}}\text{ and }i_{k+1}=i_k+1,\,i_{k+2}=i_k+2.
\end{equation*} 
However, the interval $[e, s_{i_k+1} s_{i_k} s_{i_k+2} s_{i_k+1}]$  is a subinterval of neither $\mathfrak{S}_3\times\mathfrak{B}_{\ell-3}$ nor $\mathfrak{B}_{\ell}$ by (3). Therefore, $X_w$ is smooth.

Now we prove the second statement. Note that $[e,3412]\times \mathfrak{B}^{\ell-4}$ and $\mathfrak{S}_{3}\times\mathfrak{B}^{\ell-3}$. Since we already showed that $X_{w}$ is smooth and $c(w)=1$ if and only if $[e,w]$ is isomorphic to $\mathfrak{S}_{3}\times \mathfrak{B}_{\ell-3}$ in the above, it is enough to show that $c(w)=1$. Since the poset $[e,3412]\times \mathfrak{B}_{\ell-4}$ has $\ell-1$ atoms and is of rank $\ell$,
if $[e,w]$ is isomorphic to $[e,3412]\times \mathfrak{B}_{\ell-4}$, then we get $c(w)=1$. This proves the proposition.
\end{proof} 

Therefore, the first five statements in Theorems~\ref{thm:smooth} and~\ref{thm:singular} are equivalent.

Now we determine the combinatorial type of the Bruhat interval polytope $Q_{e,w}$ when $c(w)=1$.
The combinatorial aspects of Bruhat interval polytopes are well-studied in~\cite{ts-wi}. Every face of a Bruhat interval polytope is itself a Bruhat interval polytope. However, for a subinterval $[x,y]$ of an interval $[v,w]$, $Q_{x,y}$ may not be a face of $Q_{v,w}$. For a subinterval $[x,y]$ of $[v,w]$, we introduce a directed graph $G_{x,y}^{v,w}$ which will be used to determine whether $Q_{x,y}$ is a face of $Q_{v,w}$.

Let $v\leq w$ in $\mathfrak{S}_n$. For $u\in [v,w]$, we define
\begin{align*}
\overline{T}(u,[v,w])&\coloneqq\{(i,j)\in T\mid u< u(i,j)\leq w,\,\ell(u(i,j))-\ell(u)=1\}, \\
\underline{T}(u,[v,w])&\coloneqq\{(i,j)\in T\mid v\leq u(i,j)< u,\,\ell(u)-\ell(u(i,j))=1\}, 
\end{align*} where $T$ is the set of transpositions in $\mathfrak{S}_n$, see~\eqref{eq:def_of_T}.
We first construct a labelled graph $G_{x,y}$ on $[n]=\{1,\dots,n\}$ having an edge between the vertices $a$ and $b$ if and only if $(a,b)\in \overline{T}(x,[x,y])$.

\begin{theorem}\cite{ts-wi}\label{thm:face-graph}
Let $[x,y]\subseteq [v,w]$. We define the graph $G_{x,y}^{v,w}$ as follows:
\begin{enumerate}
\item The vertices of $G_{x,y}^{v,w}$ are $\{1,2,\dots,n\}$, with vertices $i$ and $j$ identified if they are in the same connected component of the graph $G_{x,y}$.
\item There is a directed edge $i\to j$ if $(i,j)\in \overline{T}(y,[v,w])$.
\item There is a directed edge $j\to i$ if $(i,j)\in \underline{T}(x,[v,w])$.
\end{enumerate}
Then the Bruhat interval polytope $Q_{x,y}$ is a face of the Bruhat interval polytope $Q_{v,w}$ if and only if the graph $G_{x,y}^{v,w}$ is a directed acyclic graph.
\end{theorem}
\begin{remark}
The above theorem used Lemma 4.18 in~\cite{ts-wi}. Unfortunately, the proof of the lemma had a gap. Recently, the gap has been corrected by Caselli, D’Adderio and Marietti, see \cite[Remarks 5.5]{CDM}.
\end{remark}

\begin{example}
Let $w=1432$, $x=1324$, and $y=1342$. Then we have $\overline{T}(x,[x,y])=\{(3,4)\}$, 
$\overline{T}(y,[e,w])=\{(2,3)\}$, and
$\underline{T}(x,[e,w])=\{(2,3)\}$. Hence $G_{x,y}^{e,w}$ is not acyclic, and thus $Q_{x,y}$ is not a face of $Q_{e,w}$. See Figure~\ref{fig:ex-graph}.
\end{example}

\begin{figure}[ht]
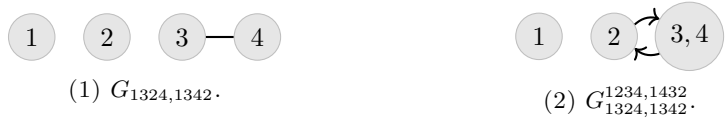

\begin{subfigure}[c]{0.49\textwidth}
\centering
 \tikz{
\node[draw=gray!50,fill=gray!20] (a) at (0,0) [circle] {$1$};
\node[draw=gray!50,fill=gray!20] (b) at (1,0) [circle] {$2$};
\node[draw=gray!50,fill=gray!20] (c) at (2,0) [circle] {$3$};
\node[draw=gray!50,fill=gray!20] (d) at (3,0) [circle] {$4$};
\draw (c) edge[thick] (d);
}
\subcaption{$G_{1324,1342}$.}
\end{subfigure}
\begin{subfigure}[c]{0.49\textwidth}
\centering
 \tikz{
\node[draw=gray!50,fill=gray!20] (a) at (0,0) [circle] {$1$};
\node[draw=gray!50,fill=gray!20] (b) at (1,0) [circle] {$2$};
\node[draw=gray!50,fill=gray!20] (c) at (2,0) [circle] {$3,4$};
\draw (c) edge[->,thick, bend left] (b);
\draw (b) edge[->,thick, bend left] (c);
}
\subcaption{$G_{1324,1342}^{1234,1432}$.}
\end{subfigure}
\caption{Examples of $G_{x,y}$ and $G_{x,y}^{v,w}$.}
\label{fig:ex-graph}
\end{figure}

We will use the following fact.  Let $a,b\in \mathfrak{S}_n$ which do not have $s_r$ in their supports.  Then we have that
\begin{equation}\label{eq:equiv}
a\le b\quad\iff\quad s_ra\le s_rb\quad\iff\quad as_r\le bs_r.
\end{equation}
In particular, 
\begin{equation}\label{eq:length}
\ell(b)-\ell(a)=\ell(s_rb)-\ell(s_ra)=\ell(bs_r)-\ell(as_r).
\end{equation}
Hence
\begin{equation}\label{eq:upper}
\overline{T}(a,[a,b])=\overline{T}(s_ra,[s_ra,s_rb]).
\end{equation}
For $\widehat{w}=s_rw$ or $ws_r$, in order to see which subinterval of $[v,\widehat{w}]$ gives a face of $Q_{v,\widehat{w}}$, we prepare the following lemma.
\begin{lemma}\label{lemm-1}
Let $u,v,w\in \mathfrak{S}_n$ with $v\le u\le w$ and $s_r\notin \supp(w)$.  Then 
\begin{enumerate}[itemsep=3pt]
\item $\underline{T}(u,[v,s_rw])=\underline{T}(u,[v,w])$ and $\underline{T}(u,[v,ws_r])=\underline{T}(u,[v,w])$; 
\item $\overline{T}(u,[v,s_rw])=\overline{T}(u,[v,w])\cup\{(u^{-1}(r),u^{-1}(r+1))\}$;
\item $\overline{T}(u,[v,ws_r])=\overline{T}(u,[v,w])\cup\{(r,r+1)\}$; 
\item  $\overline{T}(s_ru,[v,s_rw])=\overline{T}(u,[v,w])$; 
\item $\underline{T}(s_ru,[v,s_rw])=\underline{T}(u,[v,w])\cup\{(u^{-1}(r),u^{-1}(r+1))\}$;
\item $\overline{T}(us_r,[v,ws_r])=\{(s_r(i),s_r(j))\mid (i,j)\in\overline{T}(u,[v,w])\}$; and 
\item $\underline{T}(us_r,[v,ws_r])=\{(s_r(i),s_r(j))\mid (i,j)\in\underline{T}(u,[v,w])\}\cup\{(r,r+1)\}$. 
\end{enumerate}\end{lemma}
\begin{proof}
The first statement is clear from the definition of $\underline{T}(u,[v,w])$ and the assumption $s_r \notin \supp(w)$. In (2)$\sim$(5), the inclusion~($\supseteq$) is clear and hence it is enough to check the inclusion~($\subseteq$). 

In (2) and (3), for $(i,j)\in\overline{T}(u,[v,s_rw])$ (respectively, $(i,j)\in\overline{T}(u,[v,ws_r])$), if $s_r\in\supp(u(i,j))$, then $u(i,j)=s_{r}u=u(u^{-1}(r),u^{-1}(r+1))$ (respectively, $u(i,j)=us_r=u(r,r+1)$); otherwise, $(i,j)\in \overline{T}(u,[v,w])$. This proves (2) and (3). 

In (4) and (5), we use~\eqref{eq:equiv} and~\eqref{eq:length}.  For $(i,j)\in\overline{T}(s_{r}u,[v,s_{r}w])$ in (4), since $s_ru<s_ru(i,j)\leq s_rw$ and $\ell(s_ru(i,j))-\ell(s_ru)=1$,
we have $u < u(i,j)\leq w$ and $\ell(u(i,j))-\ell(u)=1$. This proves that $(i,j)\in\overline{T}(u,[v,w])$.  In~(5), if $(i,j)\in\underline{T}(s_{r}u,[v,s_{r}w])$, then  $s_{r}u(i,j)=u$ or $s_{r}v\leq s_{r}u(i,j) <s_{r}u$. Note that $s_ru(u^{-1}(r),u^{-1}(r+1))=u$. Thus if $s_{r}u(i,j)=u$, then $(i,j)=(u^{-1}(r),u^{-1}(r+1))$; otherwise, $(i,j)\in \underline{T}(u,[v,w])$.

Let us prove (6). 
It follows from~\eqref{eq:equiv} and~\eqref{eq:length} that there exists $z\in [us_r, ws_r]$ such that $\ell(z)-\ell(us_r)=1$ if and only if there exists $z'\in [u,w]$ such that $z=z's_r$ and $\ell(z')-\ell(u)=1$. Here, $z'$ is of the form
$z'=u(i,j)$ for some $(i,j)\in \overline{T}(u,[v,w])$. Since $z=u(i,j)s_r=us_r(s_r(i),s_r(j))$, we get
$$\overline{T}(us_r,[v,ws_r])\supseteq \{(s_r(i),s_r(j))\mid (i,j)\in\overline{T}(u,[v,w])\}.$$
 On the other hand, if $z's_r=us_r(i',j')$ for $(i',j')\in \overline{T}(us_r,[v,ws_r])$, then $z'=us_r(i',j')s_r=u(s_r(i'),s_r(j'))$ and $z’\leq w$. Hence there exists $(i,j)\in \overline{T}(u,[v,w])$ such that $(i’,j’)=(s_r(i),s_r(j))$.
This proves (6).

Finally, we prove (7). If $(i,j)\in \underline{T}(us_{r},[v,ws_{r}])$, then $us_{r}(i,j)=u$ or $vs_{r}\leq us_{r}(i,j)<us_{r}$. If  $us_{r}(i,j)=u$, then $(i,j)=(r,r+1)$. If $vs_{r}\leq us_{r}(i,j)<us_{r}$, then there exists  $u'\in [v,w]$ such that $us_{r}(i,j)=u’s_{r}$ and $u'< u$. Then $u’=us_{r}(i,j)s_{r}=u(s_{r}(i),s_{r}(j))$. Note that $$\ell(u)-\ell(u')=\ell(us_r)-\ell(u's_r)=\ell(us_r)-\ell(us_r(i,j))=1.$$ Hence $(s_{r}(i),s_{r}(j))\in\underline{T}(u,[v,w])$. Thus, we get
$$\underline{T}(us_{r},[v,ws_{r}])\subseteq \{(s_{r}(i),s_{r}(j))\mid (i,j)\in\underline{T}(u,[v,w])\}\cup\{(r,r+1)\}.$$
On the other hand, $(r,r+1)\in \underline{T}(us_{r},[v,ws_{r}])$ clearly. For $(i,j)\in \underline{T}(u,[v,w])$, since $v\leq u(i,j)<u$ and $\ell(u)-\ell(u(i,j))=1$, we get $vs_r\leq u(i,j)s_r<us_r$ and $\ell(us_r)-\ell(u(i,j)s_r)=1$ by~\eqref{eq:equiv} and~\eqref{eq:length}. Note that $u(i,j)s_r=us_r(s_r(i),s_r(j))$. Therefore, $(s_r(i),s_r(j))\in \underline{T}(us_r,[v,ws_r]).$ This proves (7).
\end{proof}

\begin{proposition}\label{prop:product}
For $v,w$ in $\mathfrak{S}_n$ with $v\leq w$, if $s_r\not\in\supp(w)$, then both $Q_{v,s_rw}$ and $Q_{v,ws_r}$ are combinatorially equivalent to the polytope $Q_{v,w}\times I$.
\end{proposition}
\begin{proof}

Note that for $u\in\mathfrak{S}_n$ with $u\le s_rw$ (respectively, $u\le ws_r$), if $s_r\in \supp(u)$, then there is a unique $u'$ such that $s_ru'=u$ (respectively, $u's_r=u$)  because $s_r\notin \supp(w)$.  We set
\[
\tilde{u}=\begin{cases} u \quad& \text{if $s_r\notin\supp(u)$},\\
u'\quad&\text{if $s_r\in\supp(u)$}\end{cases}\qquad\text{and}\qquad \widehat{w}=s_rw\text{ or } ws_r.
\]
Then what we have to prove is that for any $[x,y]\subseteq [v,\widehat{w}]$,
\begin{quote} 
$Q_{x,y}$ is a face of  $Q_{v,\widehat{w}}$ 
if and only if $Q_{\tilde{x},\tilde{y}}$ is a face of $Q_{v,w}$,
\end{quote}
which is equivalent to 
\begin{quote}
$(*)$\qquad  $G_{x,y}^{v,\widehat{w}}$
is acyclic if and only if $G_{\tilde{x},\tilde{y}}^{v,w}$ is acyclic
\end{quote}
by Theorem~\ref{thm:face-graph}. We will prove $(*)$ in the following. 

Note that for each $[x,y]\subseteq [v,\widehat{w}]$, there are three possibilities:
\begin{enumerate}
\item[(i)] $(\tilde{x},\tilde{y})=(x,y)$, i.e., $s_r\not\in\supp(y)$,
\item[(ii)] $(\tilde{x},\tilde{y})=(x',y')$, i.e., $s_r\in\supp(x)$, and
\item[(iii)] $(\tilde{x},\tilde{y})=(x,y')$, i.e., $s_r\not\in\supp(x)$ but $s_r\in\supp(y)$.
\end{enumerate}
Since $s_r\not\in \supp(w)$, the graph $G_{\tilde{x},\tilde{y}}$ has no directed edge between $[r]$ and $[n]\setminus[r]$. We prove $(*)$ by showing that $G_{x,y}^{v,\widehat{w}}$ is the directed graph $G_{\tilde{x},\tilde{y}}^{v,w}$ with one directed edge added between $[r]$ and $[n]\backslash [r]$ in (i) and (ii)  and with a pair of vertices in $[r]$ and $[n]\backslash [r]$ identified in (iii). 

\medskip
\noindent\underline{\textbf{Case 1: $\widehat{w}=s_rw$.}}\qquad
Note that in cases (i) and (ii), $\overline{T}(x,[x,y])=\overline{T}(\tilde{x},[\tilde{x},\tilde{y}])$ by~\eqref{eq:upper}. Hence the vertex set of $G_{x,y}^{v,\widehat{w}}$ is the same as that of $G_{\tilde{x},\tilde{y}}^{v,w}$. 
In case (i), it follows from (1) and (2) of Lemma~\ref{lemm-1} that 
\[\underline{T}(x,[v,s_rw])=\underline{T}(\tilde{x},[v,w])
\]
and
\[
\overline{T}(y,[v,s_rw])=\overline{T}(\tilde{y},[v,w])\cup\{(y^{-1}(r),y^{-1}(r+1))\}.
\]
Hence the graph $G_{x,y}^{v,\widehat{w}}$ is obtained from $G_{\tilde{x},\tilde{y}}^{v,w}$ by adding a directed edge from $y^{-1}(r)\in[r]$ to $y^{-1}(r+1)\in[n]\setminus[r]$.
In case (ii), it follows from (4) and (5) of Lemma~\ref{lemm-1} that
\[
\overline{T}(y,[v,s_rw])=\overline{T}(\tilde{y},[v,w])
\]
 and 
 \[
 \underline{T}(x,[v,s_rw])=\underline{T}(\tilde{x},[v,w])\cup\{((\tilde{x})^{-1}(r),(\tilde{x})^{-1}(r+1))\}.
\]
Hence the graph $G_{x,y}^{v,\widehat{w}}$ is obtained from $G_{\tilde{x},\tilde{y}}^{v,w}$ by adding a directed edge from $(\tilde{x})^{-1}(r+1)\in[n]\setminus[r]$ to $(\tilde{x})^{-1}(r)\in[r]$.
In case (iii), by (1) and (4) of Lemma~\ref{lemm-1}, we have
\[
\underline{T}(x,[v,s_rw])=\underline{T}(\tilde{x},[v,w])\text{ and } \overline{T}(y,[v,s_rw])=\overline{T}(\tilde{y},[v,w]).
\]
Since
\[
\overline{T}(x,[x,y])=\overline{T}(\tilde{x},[\tilde{x},\tilde{y}])\cup\{(x^{-1}(r),x^{-1}(r+1))\}
\]
by Lemma~\ref{lemm-1}(2),
the graph $G_{x,y}^{v,\widehat{w}}$ is obtained from $G_{\tilde{x},\tilde{y}}^{v,w}$ by identifying the vertices $x^{-1}(r)\in[r]$ and $x^{-1}(r+1)\in[n]\setminus[r]$. This proves Case 1.

\medskip
\noindent\underline{\textbf{Case 2: $\widehat{w}=ws_r$.}}\qquad
In case (i), since $\overline{T}(x,[x,y])=\overline{T}(\tilde{x},[\tilde{x},\tilde{y}])$, the vertex set of $G_{x,y}^{v,\widehat{w}}$ is the same as that of $G_{\tilde{x},\tilde{y}}^{v,w}$. 
It follows from (1) and (3) of Lemma~\ref{lemm-1} that
\[\underline{T}(x,[v,ws_r])=\underline{T}(\tilde{x},[v,w])
\] and
\[
\overline{T}(y,[v,ws_r])=\overline{T}(\tilde{y},[v,w])\cup\{(r,r+1)\}.
\] Hence
the graph $G_{x,y}^{v,\widehat{w}}$ is obtained from $G_{\tilde{x},\tilde{y}}^{v,w}$ by adding a directed edge from $r$ to $r+1$.
In case (ii), 
since $\tilde{x}(i,j)s_r=\tilde{x}s_r(s_r(i),s_r(j))=x(s_r(i),s_r(j))$, we get
$$\overline{T}(x,[x,y])=\{(s_r(i),s_r(j))\mid (i,j)\in \overline{T}(\tilde{x},[\tilde{x},\tilde{y}])\}$$
similarly to the proof of (6) in Lemma~\ref{lemm-1}.
Hence $G_{x,y}$ is obtained from $G_{\tilde{x},\tilde{y}}$ by interchanging the labelling of the vertices $r$ and $r+1$.
By (6) and (7) of Lemma~\ref{lemm-1}, we have 
\[
\overline{T}(y,[v,ws_r])=\{(s_r(i),s_r(j))\mid (i,j)\in\overline{T}(\tilde{y},[v,w])\}
\]
 and 
 \[
 \underline{T}(x,[v,ws_r])=\{(s_r(i),s_r(j))\mid (i,j)\in\underline{T}(\tilde{x},[v,w])\}\cup\{(r,r+1)\}.
\]
Thus
the graph $G_{x,y}^{v,\widehat{w}}$ is obtained from $G_{\tilde{x},\tilde{y}}^{v,w}$ by 
interchanging the labelling of the vertices $r$ and $r+1$ and then adding a directed edge from $r+1$ to $r$. In case (iii), it follows from (1) and (6) of Lemma~\ref{lemm-1} that
\[\underline{T}(x,[v,ws_r])=\underline{T}(\tilde{x},[v,w])\] and 
\[\overline{T}(y,[v,ws_r])=\{(s_r(i),s_r(j))\mid (i,j)\in\overline{T}(\tilde{y},[v,w])\}.\]
Note that 
$$\overline{T}(x,[x,y])=\overline{T}(\tilde{x},[\tilde{x},\tilde{y}])\cup\{(r,r+1)\}$$
by Lemma~\ref{lemm-1}(3). Hence the graph $G_{x,y}$ is obtained from $G_{\tilde{x},\tilde{y}}$ by adding an edge between the vertices $r$ and $r+1$. Therefore, the graph $G_{x,y}^{v,\widehat{w}}$ is obtained from $G_{\tilde{x},\tilde{y}}^{v,w}$ by identifying the vertices $r$ and $r+1$.
\end{proof}

The above proposition implies that $Q_{e,w}$ is combinatorially equivalent to the cube $I^{\ell(w)}$ if $X_w$ is of complexity zero, i.e., a reduced decomposition of~$w$ consists of distinct letters.

\begin{theorem}\label{thm:poly-comb}
For a permutation $w$ in $\mathfrak{S}_n$, the following hold:
\begin{enumerate}
\item the Schubert variety  $X_w$  is smooth and of complexity one if and only if $Q_{e,w}$ is combinatorially equivalent to the polytope $\Perm_2\times I^{\ell(w)-3}$, and
\item the Schubert variety~$X_{w}$ is singular and of complexity one if and only if $Q_{e,w}$ is combinatorially equivalent to the polytope $Q_{e,3412}\times I^{\ell(w)-4}$.
\end{enumerate}
\end{theorem}
\begin{proof} 
Since both the polytopes $\Perm_2\times I^{\ell(w)-3}$ and $Q_{e,3412}\times I^{\ell(w)-4}$ are $\ell(w)-1$ dimensional, we get $c(w)=1$. Hence it suffices to prove the ‘only if’ part in each statement.

If a Schubert variety $X_w$ is smooth and of complexity one, then $w$ has a reduced decomposition $\underline{w}$ of the form:
\begin{equation*}
\underline{w}=s_{i_1}\cdots s_{i_{k-1}} s_{i_{k}} s_{i_{k+1}} s_{i_{k}} s_{i_{k+2}}\cdots s_{i_{\ell-1}}\text{ and }i_{k+1}=i_k+1.
\end{equation*}
Hence the polytope $Q_{e,w}$ is combinatorially equivalent to $Q_{e,s_{i_{k}}s_{i_{k}+1}s_{i_{k}}}\times I^{\ell(w)-3}\cong\Perm_2\times I^{\ell(w)-3}$ by Lemma~\ref{lem:repeated-elements} and Proposition~\ref{prop:product}.
If $X_{w}$ is singular and of complexity one, then $w$ has a reduced decomposition $\underline{w}$ of the form:
\begin{equation*}
\underline{w}=s_{i_1}\cdots s_{i_{k-1}} s_{i_{k+1}} s_{i_k} s_{i_{k+2}} s_{i_{k+1}} s_{i_{k+3}}\cdots s_{i_{\ell-1}}\text{ and }i_{k+1}=i_k+1,\,i_{k+2}=i_k+2.
\end{equation*}
Hence $Q_{e,w}$ is combinatorially equivalent to $Q_{e,s_{i_{k}+1}s_{i_{k}}s_{i_{k}+2}s_{i_{k}+1}}\times I^{\ell(w)-4}\cong Q_{e,3412}\times I^{\ell(w)-4}$ by Lemma~\ref{lem:repeated-elements} and Proposition~\ref{prop:product}.
\end{proof}

The first statement in the above theorem shows the equivalence between $(1^{\prime})$ and  $(6^{\prime})$ in Theorem~\ref{thm:smooth}, and hence it proves Theorem~\ref{thm:smooth}.
The second statement in the above theorem shows the equivalence between $(1^{\prime\prime})$ and  $(6^{\prime\prime})$ in Theorem~\ref{thm:singular}, and hence it proves Theorem~\ref{thm:singular}.

Note that for a permutation $w$ in $\mathfrak{S}_{n}$, the following hold:
\begin{enumerate}
\item $w$ avoids the patterns $3412$ and $4231$ if and only if $w^{-1}$ avoids those patterns, and
\item $[321;3412](w)=1$ if and only if $[321;3412](w^{-1})=1$.
\end{enumerate}
Hence $X_{w}$ is smooth and of complexity one if and only if $X_{w^{-1}}$ is smooth and of complexity one.
\begin{corollary}
For $w\in\mathfrak{S}_{n}$, if $[321;3412](w)=1$, then the Bruhat interval polytopes $Q_{e,w}$ and $Q_{e,w^{-1}}$ are combinatorially equivalent.
\end{corollary}
Therefore, for a Schubert variety $X_{w}$ of complexity one, $X_{w}$ is smooth if and only if the moment polytope $Q_{e,w^{-1}}$ is simple.

%%%%%%%%%%%%%%%%%%%%%%%%%%%%%%%%%%%%%%%%%%%%
\section{Flag Bott manifolds}\label{sec:flag-bott}
Like as a Bott--Samelson variety is diffeomorphic to a Bott manifold with a higher rank torus action, a flag Bott--Samelson variety is diffeomorphic to a \defi{flag Bott manifold} with a higher rank torus action. Whereas a Bott manifold is a toric variety, a flag Bott manifold is not a toric variety in general, but it becomes a GKM manifold, see~\cite{KLSS18}.
We have previously observed in Theorem~\ref{thm:main1} that there is an isomorphism between a smooth Schubert variety and a flag Bott--Samelson variety. Using the diffeomorphism between a flag Bott--Samelson variety and a certain flag Bott manifold, we will provide a formula for the  cohomology ring of a smooth Schubert variety of complexity one.

Recall that for a holomorphic vector bundle $E$ over a complex manifold $M$, there is an associated flag-bundle $\flag(E)$ obtained from $E$ by replacing each fiber $E_p$ by the full flag manifold $\flag(E_p)$.

\begin{definition}[{\cite[Deifnition~2.1]{KLSS18}}]
A \defi{flag Bott tower} is an iterated flag-bundle
\begin{equation*}
\mathcal{F}_{r}\overset{\pi_r}\longrightarrow \mathcal{F}_{r-1}\overset{\pi_{r-1}}\longrightarrow\cdots \overset{\pi_2}\longrightarrow \mathcal{F}_1 \overset{\pi_1}\longrightarrow \mathcal{F}_0=\{\text{a point}\}
\end{equation*} of manifolds $\mathcal{F}_k=\flag\left(\underline{\C}\oplus\bigoplus_{m=1}^{n_k} \xi_k^{(m)}\right)$,
 where $\xi_k^{(m)}$ is a holomorphic line bundle over $\mathcal{F}_{k-1}$ for each $1\leq m \leq n_{k}$ and $1\leq k\leq r$. Each $\mathcal{F}_{k}$ is called a \defi{flag Bott manifold} (of height $k$).
\end{definition}
Because we are considering an iterated full flag-bundle, there is an isomorphism
 \[
	\psi \colon \Z^{n_1} \times \cdots \times \Z^{n_r} \to \textup{Pic}(F_r).
	\]
Therefore, there is a sequence
$$\left(\a_{1,k}^{(m)},\a_{2,k}^{(m)},\dots,\a_{k-1,k}^{(m)}\right)$$ of integer vectors with $\a_{j,k}^{(m)}\in\Z^{n_j}$ which maps to $\xi_k^{(m)}$ via $\psi$. Hence a set  $\{\a_{j,k}^{(m)}\}_{1 \leq m \leq n_k, 1 \leq j<k\leq r}$ of integer vectors determines a flag Bott manifold.

\begin{theorem}\cite[Theorem~{4.10}]{FLS} \label{thm:flag-bott-degeneration}
Let $\mathcal{I}=(I_1,\ldots,I_r)$ be a sequence of subsets of~$[n]$, where $I_k=\{u_{k,1},\ldots,u_{k,n_k}\}$ for $1\leq k\leq r$. Assume that each $I_k$ is an interval. Then the flag Bott--Samelson variety $Z_{\mathcal{I}}$ is diffeomorphic to a flag Bott manifold~$\mathcal{F}_r$ determined by the vectors $$\a_{j,k}^{(m)}=(\a_{j,k}^{(m)}(1),\ldots,\a_{j,k}^{(m)}(n_j))\in \Z^{n_j}\qquad (1\leq m\leq n_k\text{ and }1\leq j<k\leq r)$$ where
\begin{equation*}\label{eq:flag-bott-integers}
\a_{j,k}^{(m)}(p)=\langle \mathbf{e}_{u_{j,p}} -\mathbf{e}_{u_{j,n_j+1}}, \mathbf{e}_{u_{k,m}} - \mathbf{e}_{u_{k,n_k+1}} \rangle
\end{equation*}
for $1\leq p\leq n_j$. Here, $\mathbf{e}_{u_{k,1}},\dots,\mathbf{e}_{u_{k,n_k+1}}$ are the standard basis vectors of $\R^{n_k+1}$ for $1 \leq k \leq r$.
\end{theorem}

\begin{theorem}\label{thm:main2}
Let $X_w$ be a smooth Schubert variety of complexity one. Then $X_w$ is diffeomorphic to a flag Bott  manifold of height $\ell(w)-2$

\begin{tikzcd}
{\mathcal{F}_{\ell(w)-2}}\arrow[r, "\pi_{\ell(w)-2}"]  & {\mathcal{F}_{\ell(w)-3}} \arrow[r, "\pi_{\ell(w)-3}"] & {\cdots} \arrow[r,"\pi_2"] & {\mathcal{F}_1} \arrow[r,"\pi_1"] & {\mathcal{F}_0=\{\text{a point}\}},
\end{tikzcd}
\\
where $\mathcal{F}_q\to \mathcal{F}_{q-1}$ is a $\flag(\C^3)$-bundle for some $1\leq q\leq \ell(w)-2$ and $\mathcal{F}_k\to \mathcal{F}_{k-1}$ is a $\C P^1$-bundle for every $k\neq q$. Furthermore, the iterated bundle structure is completely determined by a reduced decomposition of $w$.
\end{theorem}

\begin{proof}
Let $\ell(w)=\ell$.
From the proof of  Theorem~\ref{thm:main1}, we may assume that $X_w$ is isomorphic to a flag Bott--Samelson variety $Z_{\mathcal{I}}$, where $\mathcal I = (I_1,\dots,I_{\ell-2})$ of length $\ell-2$ consists of the sets
\[
	I_k = \begin{cases}
	\{i_k\} & \text{ if } 1 \leq k < q, \\
	\{i, i+1\} & \text{ if } k = q, \\
	\{i_{k+2}\} & \text{ if } k > q.
	\end{cases}
	\]
Note that $\underline{w}=({i_1},\dots, {i_{q-1}},i,{i+1},i,{i_{q+3}},\dots ,{i_{\ell}})$.
Thus $X_w$ is diffeomorphic to a flag Bott manifold of height $\ell-2$ by Theorem~\ref{thm:flag-bott-degeneration}. Furthermore, for each $1\leq k\leq \ell-2$, the manifold $\mathcal{F}_k$ is determined as follows. 
\begin{enumerate}
\item If $1\leq k<q$, then $n_k=1$ and $\mathcal{F}_k=\flag(\underline{\C}\oplus \xi_k)$, where $\xi_k$ is a holomorphic line bundle determined by a sequence $(\mathbf{a}_{1,k}^{(1)},\dots,\mathbf{a}_{k-1,k}^{(1)})$ of integers, where
\[
\mathbf{a}_{j,k}^{(1)} = \langle  \mathbf{e}_{i_j} - \mathbf{e}_{i_j+1} , \mathbf{e}_{i_k} - \mathbf{e}_{i_{k}+1}\rangle \in \Z.
\]
\item If $k=q$, then $n_j=1$ for $j < k = q$ and $n_k=2$ and $\mathcal{F}_k=\flag(\underline{\C}\oplus \xi_k^{(1)}\oplus \xi_k^{(2)})$, where $\xi_k^{(1)}$ and $\xi_k^{(2)}$ determined by sequences 
$( \mathbf{a}_{1,q}^{(1)},\dots,\mathbf{a}_{q-1,q}^{(1)})$ and $(\mathbf{a}_{1,q}^{(2)},\dots,\mathbf{a}_{q-1,q}^{(2)})$ of integers, respectively. Here, we have that
\[
\begin{split}
\mathbf{a}_{j,q}^{(1)}  &= \langle \mathbf{e}_{i_j}-\mathbf{e}_{i_j+1}, \mathbf{e}_{i} - \mathbf{e}_{i+2}\rangle \in \Z,\\
\mathbf{a}_{j,q}^{(2)}  &= \langle \mathbf{e}_{i_j} - \mathbf{e}_{i_j+1},\mathbf{e}_{i+1}-\mathbf{e}_{i+2}\rangle \in \Z.
\end{split}
\]
\item If $k>q$, then $n_k=1$ and $\mathcal{F}_k=\flag(\underline{\C}\oplus \xi_k)$, where $\xi_k$ is a holomorphic line bundle  determined by a sequence $(\mathbf{a}_{1,k}^{(1)},\dots,\mathbf{a}_{k-1,k}^{(1)})$ of integer vectors, where
\[
\mathbf{a}_{j,k}^{(1)}
= \begin{cases}
\langle \mathbf{e}_{i_j} - \mathbf{e}_{i_j+1}, \mathbf{e}_{i_{k+2}} - \mathbf{e}_{i_{k+2}+1}\rangle &\text{ if } j < q,\\
(\langle \mathbf{e}_{i} - \mathbf{e}_{i+2}, \mathbf{e}_{i_{k+2}} - \mathbf{e}_{i_{k+2}+1}\rangle, \langle \mathbf{e}_{i+1} - \mathbf{e}_{i+2}, \mathbf{e}_{i_{k+2}} - \mathbf{e}_{i_{k+2}+1}\rangle) 
& \text{ if } j = q, \\
\langle \mathbf{e}_{i_{j+2}} - \mathbf{e}_{i_{j+2}+1}, \mathbf{e}_{i_{k+2}} - \mathbf{e}_{i_{k+2}+1}\rangle  &\text{ if } j > q.
\end{cases}
\]
\end{enumerate}
This proves the theorem.
\end{proof}

Combining Theorem~\ref{thm:Gr-Ka} with Theorems~\ref{thm:main1} and \ref{thm:main2}, we  conclude that every smooth Schubert variety of complexity $\leq 1$ is isomorphic to a flag Bott--Samelson variety, and hence it is diffeomorphic to a flag Bott manifold whose height is determined by the length $\ell(w)$ and the complexity $c(w)$.
Therefore, using~\cite[Corollary~4.4]{KKLS20}, we provide the cohomology ring $H^\ast(X_w;\Z)$ of a smooth Schubert variety $X_w$ with $c(w)\leq 1$.
\begin{corollary}\label{cor_cohomology_ring_of_Xw}
	Let $X_w$ be a smooth Schubert variety of complexity one. Suppose that $\underline{w} = s_{i_1}\dots s_{i_{q-1}} s_{i} s_{i+1} s_{i} s_{i_{q+3}} \dots s_{i_{\ell}}$ is a reduced decomposition of $w$. Then  the cohomology ring $H^{\ast} (X_w;\Z)$ is given as follows:
	\[
	H^{\ast} (X_w;\Z) \cong \Z [ y_{j,1},\dots,y_{j,n_j+1} \mid 1 \leq j \leq \ell-2]/\langle I_1,\dots,I_{\ell-2} \rangle.
	\]
	Here, $y_{j,k}$'s are degree two elements and $I_{k}$ is the ideal given as follows.
	\[
	\begin{split}
	\text{ For }k<q:&\quad  I_k = (1-y_{k,1})(1-y_{k,2}) - \left(1-\sum_{j=1}^{k-1} \mathbf{a}_{j,k}^{(1)} y_{j,1} \right), \\
	\text{ for } k =q:& \quad I_q = (1-y_{q,1})(1-y_{q,2})(1-y_{q,3}) \\
	&\qquad\qquad- \left(1-\sum_{j=1}^{q-1} \mathbf{a}_{j,q}^{(1)}y_{j,1} \right) \left(1-\sum_{j=1}^{q-1} \mathbf{a}_{j,q}^{(2)} y_{j,1}\right), \\
	\text{ for }k > q: & \quad I_k = (1-y_{k,1})(1-y_{k,2}) \\
	&\qquad\qquad- \left(1- \sum_{\substack{1 \leq j \leq k-1, \\ j \neq q}} \mathbf{a}_{j,k}^{(1)} y_{j,1} - (\mathbf a_{q,k}^{(1)} (1) y_{q,1} + \mathbf a_{q,k}^{(1)}(2) y_{q,2}) \right).
	\end{split}
	\]
\end{corollary}
\begin{example}\label{example_coh_description}
	Suppose that $w = s_1s_2s_3s_4s_3$. Then the Schubert variety $X_w$ is smooth of complexity one. 
	By Theorem~\ref{thm:flag-bott-degeneration}, the Schubert variety $X_w$ is diffeomorphic to a flag Bott manifold of height~$3$ with $n_1 = n_2 =1$ and $n_3 = 2$ which is determined by the following integer vectors:
	\[
	\begin{split}
	\mathbf{a}_{1,2}^{(1)} &= \langle \mathbf e_1 - \mathbf e_2, \mathbf e_2 - \mathbf e_3 \rangle= -1, \\
	\mathbf{a}_{1,3}^{(1)} &= \langle \mathbf e_{1} - \mathbf{e}_2, \mathbf e_3 - \mathbf e_5 \rangle = 0, \quad  \mathbf{a}_{1,3}^{(2)} = \langle \mathbf e_1 - \mathbf e_2, \mathbf e_4 - \mathbf e_5 \rangle = 0, \\
	\mathbf{a}_{2,3}^{(1)} &= \langle \mathbf e_2 - \mathbf e_3, \mathbf e_3 - \mathbf e_5 \rangle = -1, \quad
	\mathbf{a}_{2,3}^{(2)} = \langle \mathbf e_2 - \mathbf e_3, \mathbf e_4 - \mathbf e_5 \rangle = 0.\\
	\end{split}
	\]
	Therefore, by Corollary~\ref{cor_cohomology_ring_of_Xw}, the cohomology ring of $X_w$ is 
	\[
	H^{\ast}(X_w; \Z) \cong \Z[y_{j,1},\dots,y_{y,n_j+1} \mid 1 \leq j \leq 3]/ \langle I_1,\dots,I_3 \rangle, 
	\]
	where
	\[
	\begin{split}
	I_1 &= (1-y_{1,1})(1-y_{1,2}) - 1, \\
	I_2 &= (1-y_{2,1})(1-y_{2,2}) - (1 + y_{1,1}), \\
	I_3 &= (1-y_{3,1})(1-y_{3,2})(1-y_{3,3}) - (1 + y_{2,1}). \\
	\end{split}
	\]
	This gives relations among generators:
		\[
		\begin{split}
		&y_{1,1} + y_{1,2} = y_{1,1} y_{1,2} = 0, \\
		& y_{2,1} + y_{2,2} + y_{1,1} = y_{2,1} y_{2,2} = 0, \\
		& y_{3,1} + y_{3,2} + y_{3,3} + y_{2,1} = y_{3,1}y_{3,2} + y_{3,1}y_{3,3} + y_{3,2} y_{3,3} = y_{3,1}y_{3,2} y_{3,3} = 0.\\
		\end{split}
		\]
Therefore, by setting $y_1 = y_{1,1}$, $y_{2} = y_{2,1}$, $y_{3} = y_{3,1}$, $y_4 = y_{3,2}$, we have that
\[
H^{\ast}(X_{w};\Z) \cong \Z[y_1,\dots,y_4]/I,
\]	
where $I$ is an ideal generated by 
\begin{equation*}
y_1^2, y_2(y_1+y_2), (y_2+y_3)y_3 + (y_2+y_3+y_4)y_4, y_3y_4(y_2+y_3+y_4).
\end{equation*}
\end{example}

We enclose this section by mentioning other studies on the cohomology rings of smooth Schubert varieties. Indeed, the cohomology rings of smooth Schubert varieties are studied in~\cite{VR_cohomology} (also, see~\cite{DMR}). We will demonstrate their result for a specific permutation $23541$, which is the one considered in Example~\ref{example_coh_description}.
First we recall from~\cite{DMR} that this permutation is related to a partition as follows. 
For a given partition $\lambda = (0 \leq \lambda_1 \leq \dots \leq \lambda_n)$, one may associate a permutation $w = w(1) \dots w(n)$ by the recursive rule
\[
w(i) = \max(\{ 1,\dots,\lambda_i\} \setminus \{ w(1),\dots,w(i-1)\}).
\]
For example, if $\lambda = (2,3,5,5,5)$, then the corresponding permutation is $23541$. We denote by $X_{\lambda}$ the Schubert variety given by the permutation coming from $\lambda$. The presentation for the cohomology ring of $X_{\lambda}$ is given as follows:
\begin{equation}\label{eq:coho-ring}
H^{\ast}(X_{\lambda};\Z) \cong \Z[x_1,\dots,x_n]/ \langle h_{\lambda_i - i +1}(i) \mid 1 \leq i \leq n\rangle
\end{equation}
where $h_{m}(N)$ is the complete homogeneous symmetric function:
\[
h_m(N) \coloneqq h_m(x_1,\dots,x_N) = \sum_{1 \leq i_1 \leq \dots \leq i_m \leq N}x_{i_1} \cdots x_{i_m}.
\]
For the partition $\lambda = (2,3,5,5,5)$, the ideal $J:=\langle h_{\lambda_i - i +1}(i) \mid 1 \leq i \leq n\rangle$ on the right hand side of~\eqref{eq:coho-ring} is generated by the complete homogeneous symmetric polynomials
\[
\begin{split}
h_{2}(1)  &= x_1^2, \\
h_2(2) &= x_1^2 + (x_1+x_2)x_2, \\
h_3(3)&= (x_1+x_2+x_3)x_1^2 + (x_1x_2 + x_2^2)(x_2+x_3) + (x_1+x_2+x_3)x_3^2, \\
h_2(4) &= x_1^2 + (x_1+x_2)x_2 + (x_1+x_2+x_3)x_3 + (x_1+x_2+x_3+x_4)x_4, \\
h_1(5) &=x_1+x_2+x_3+x_4+x_5.
\end{split}
\]
Moreover, we get that $\Z[x_1,\dots,x_5]/J \cong \Z[x_1,\dots,x_4]/J'$, where
\[
\begin{split}
J' &= \langle x_1^2, (x_1+x_2)x_2, (x_1+x_2+x_3)x_3^2, 
(x_1+x_2+x_3)x_3 + (x_1+x_2+x_3+x_4)x_4 \rangle. 
\end{split}
\]
By sending $y_1 \mapsto -x_1$, $y_2 \mapsto x_1+x_2$, $y_3 \mapsto x_3$, $y_4 \mapsto x_4$, we obtain an isomorphism $\Z[x_1,\dots,x_5]/I \cong \Z[y_1,\dots,y_5]/J'$ between the two cohomology ring representations. For instance, we have that
\[
(y_2+y_3)y_3 + (y_2+y_3+y_4)y_4
\mapsto (x_1+x_2+x_3)x_3 + (x_1+x_2+x_3+x_4)x_4,
\]
\[
\begin{split}
&y_3y_4(y_2+y_3+y_4) \\
&\quad \mapsto x_3x_4(x_1+x_2+x_3+x_4) \\
&\qquad =x_3((x_1+x_2+x_3)x_3 + (x_1+x_2+x_3+x_4)x_4) - (x_1+x_2+x_3)x_3^2.
\end{split}
\]

\begin{remark}
The notion of flag Bott--Samelson variety can be defined in a general Lie type, where $G$ is a simply-connected semisimple algebraic group   over~$\C$. Furthermore, Proposition~\ref{prop_multiplication_map_pI}, Corollary~\ref{cor_fBS_and_Schubert_isomorphic}, and Theorem~\ref{thm:flag-bott-degeneration} are still true in a general Lie type.
\end{remark} 

\subsection*{Acknowledgements} 
Lee was supported by IBS-R003-D1.
Masuda was supported in part by JSPS Grant-in-Aid for Scientific Research 19K03472 and a bilateral program between JSPS and RFBR.
Park was supported by the Basic Science Research Program through the National Research Foundation of Korea (NRF) funded by the Government of Korea (NRF-2018R1A6A3A11047606). This work was partly supported by Osaka City University Advanced
Mathematical Institute (MEXT Joint Usage/Research Center on Mathematics
and Theoretical Physics JPMXP0619217849).

\providecommand{\bysame}{\leavevmode\hbox to3em{\hrulefill}\thinspace}
\providecommand{\MR}{\relax\ifhmode\unskip\space\fi MR }
\providecommand{\MRhref}[2]{%
	\href{http://www.ams.org/mathscinet-getitem?mr=#1}{#2}
}
\providecommand{\href}[2]{#2}

\end{document}